\documentclass{amsart}[12pt font]
\pdfoutput=1
\author{Tamanna Chatterjee}
\date{\today}
\title{Study of Parity Sheaves arising from Graded Lie algebras}
\usepackage{amssymb}
\usepackage{amsmath}
\usepackage{amsfonts}
\usepackage{mathrsfs}
\usepackage{mathtools}
\usepackage{tikz-cd}
\usepackage{placeins}
\theoremstyle{plain}
\newtheorem{theorem}{Theorem}[section]
\newtheorem{corollary}[theorem]{Corollary}
\newtheorem{lemma}[theorem]{Lemma}
\newtheorem{proposition}[theorem]{Proposition}
\theoremstyle{definition}
\newtheorem{remark}[theorem]{Remark}
\newtheorem{conjecture}[theorem]{Conjecture}
\newtheorem{assume}[theorem]{Assumption}
\newtheorem{definition}[theorem]{Definition}
\newtheorem{example}[theorem]{Example}

\newcommand{\mf}{\mathfrak}
\newcommand{\mc}{\mathcal}
\newcommand{\mb}{\mathbb}
\newcommand{\ms}{\mathscr}
\newcommand{\rg}{R\Gamma}
\newcommand{\eal}{\end{align*}}
\newcommand{\al}{\alpha}
\newcommand{\bal}{\begin{align*}}
\DeclareMathOperator{\res}{Res^\mathfrak{g}_{\mathfrak{p}}}
\DeclareMathOperator{\ind}{Ind^\mathfrak{g}_{\mathfrak{p}}}
\DeclareMathOperator{\Ind}{Ind}
\DeclareMathOperator{\inn}{Ind}
\DeclareMathOperator{\p}{Perv}
\DeclareMathOperator{\f}{For}
 \DeclareMathOperator{\stab}{Stab}
 
 \DeclareMathOperator{\Res}{Res}
\DeclareMathOperator{\loc}{Loc_f}

\DeclareMathOperator{\lie}{Lie}
\DeclareMathOperator{\h}{H}
\DeclareMathOperator{\Hom}{Hom}
\DeclareMathOperator{\lc}{Loc}

\DeclareMathOperator{\cu}{cusp}

\DeclareMathOperator{\i'}{\mathscr{I}(\mathfrak{l}_n)^{cusp}}
\DeclareMathOperator{\IC}{\mathcal{IC}}
\DeclareMathOperator{\E}{\mathcal{E}}
\begin{document}
\maketitle
\begin{abstract}
    Let $G$ be a complex, connected, reductive, algebraic group, and   $\chi:\mb{C}^\times \to G$ be  a  fixed cocharacter that defines a grading on $\mf{g}$, the Lie algebra of $G$. Let $G_0$ be the centralizer of $\chi(\mb{C}^\times)$. In this paper, we  study $G_0$-equivariant parity sheaves on $\mf{g}_n$, under some assumptions on  the field $\Bbbk$ and the group $G$. The assumption on $G$ holds for $GL_n$ and for any $G$, it recovers results of Lusztig\cite{Lu} in characteristic $0$. The main result is that every parity sheaf occurs as a direct summand of the parabolic induction  of some cuspidal pair.
\end{abstract}

\section{Introduction}
Let $G$ be a complex, connected, reductive, algebraic group and   $\chi: \mb{C}^\times \to G$ be a fixed  cocharacter. Let $G_0$ be the centralizer of $\chi(\mb{C}^\times)$ and $\mf{g}_n\subset \mf{g}$ be the subspace such that  $Ad(\chi(t))$ acts on it by $t^n$ times identity. We are particularly interested in studying  the derived category of $G_0$-equivariant perverse sheaves on $\mf{g}_n$, denoted by $D^b_{G_0}(\mf{g}_n,\Bbbk)$. Here $\Bbbk$ is a field of positive characteristic. The simple perverse sheaves on $\mf{g}_n$ are indexed by $(\mc{O},\mc{L})$, where $\mc{O}$ is a $G_0$-orbit contained in $\mf{g}_n$ and $\mc{L}$ is an irreducible $G_0$-equivariant $\Bbbk$-local system on $\mc{O}$. We denote this set of pairs by $\ms{I}(\mf{g}_n,\Bbbk)$. We   define in section \ref{sec2},  the subset of all cuspidal pairs,  $\ms{I}(\mf{g}_n,\Bbbk)^{\cu}\subset \ms{I}(\mf{g}_n,\Bbbk)$.  We denote the simple perverse sheaf associated to $(\mc{O},\mc{L})$ by $\mc{IC}(\mc{O},\mc{L})$. Motivated by  applications to affine Hecke algebras, Lusztig has worked in $\Bbbk=\mb{C}$ and has proved in \cite{Lu} that every simple perverse sheaf is a direct summand of the parabolic induction of some cuspidal pair. But in positive characteristics, this result is not true.

Following the pattern from other works in modular representation theory, often the appropriate replacement for ``semisimple complex" is ``parity complex''. Parity sheaves are classified as the class of constructible complexes on some stratified varieties, where the strata satisfies  some cohomology vanishing properties \cite{parity}.  We denote the parity sheaf associated to the pair $(\mc{O},\mc{L})$ by $\mc{E}(\mc{O},\mc{L})$. So the most fundamental question that arises is if they exist on $\mf{g}_n$.  Before going into that question we make some assumptions on the field coefficient.  We assume that the characteristic $l$ of $\Bbbk$ is  ``pretty good" and the field is ``big-enough" for $G$. Both the definitions are given in subsection \ref{subsec1.6}.  A pair $(\mc{O},\mc{L})\in \ms{I}(\mf{g}_n,\Bbbk)$ is said to be clean if $\mc{IC}(\mc{O},\mc{L})$ has vanishing stalks on $\bar{\mc{O}}-\mc{O}$. 

Under the above assumptions, we  assume Mautner's cleanness conjecture (conjecture \ref{con}) is true, this plays an important role in the proofs of  the main  theorems of this paper. Mautner's conjecture already holds when the characteristic $l$ does not divide the order of the Weyl group of the group $G$ or  if every irreducible factor of the root system of $G$ is either of type $A,B_4,C_3,D_5$ or of exceptional types.  We make another conjecture (conjecture \ref{2.6(c)}), at the end of section \ref{sec2} that on  the nilpotent cone, parabolic induction preserves the parity of any cuspidal pair on a Levi subgroup. This conjecture is known to be true for $GL_n$ and in characteristic $0$ it is true for any group $G$. In section \ref{sec10} we will show that this conjecture is also true for $Sp_4$ and $SL_4$. The following are the main results of this paper.

\begin{enumerate}\item For any cuspidal pair,  $(\mc{O},\mc{L})\in \ms{I}(\mf{g}_n,\Bbbk)^{\cu}$, $\mc{IC}(\mc{O},\mc{L})$ is clean and so $\mc{IC}(\mc{O},\mc{L})=\mc{E}(\mc{O},\mc{L})$.
\item
Parabolic induction takes parity complexes to parity complexes.
\item
For any pair $(\mc{O},\mc{L})\in \ms{I}(\mf{g}_n,\Bbbk)$, $\mc{E}(\mc{O},\mc{L})$ exists and is a direct summand of the parabolic induction of some cuspidal pair.	
\end{enumerate}

The proof of the existence of parity sheaves for the space of  quiver representation of type $A,D,E$ is given in \cite{Ma}, using some other methodology. For some of these cases the space coming from the quiver representation is the same space that we study here. So for these cases existence of parity sheaf has already been proved for $\mf{g}_n$.
\subsection{Outline}In section \ref{sec2} we build the  necessary background, assumptions and notations. 
Section \ref{sec3} contains the lemmas on the varieties having $\mb{C}^\times$ action on it. In section \ref{sec4}, we define $\ind$ and $\res$ in the graded setting and prove  
 the existence of parity sheaves for cuspidal pairs. In section \ref{sec5}, we redefine both $\ind$ and $\inn^G_P$ for cuspidal pairs and in theorem \ref{th5},  we prove that the parity condition is preserved for cuspidal pairs. In section \ref{sec6}, 
we prove that the parity sheaf exists for a general pair $(\mc{O},\mc{L})\in \ms{I}(\mf{g}_n,\Bbbk)$ and  in  \ref{sec8}, we prove that the parabolic induction preserves parity for a general pair. In section \ref{sec10}, we compute some examples.
\section*{Acknowledgements}
The author owes an enormous amount of debt to her advisor Pramod Achar for his unconditional help throughout the process. This work would not be possible without his unlimited patience and encouragement. The author received support from NSF Grant Nos. DMS-1500890 and DMS-1802241.
\tableofcontents
\section{Background}\label{sec2}
Let $\Bbbk$ be a field of characteristic $l>0$. We  consider  sheaves with coefficients in $\Bbbk$. The varieties we  work on will be over $\mathbb{C}$. Let  $G$  be a connected, reductive, algebraic group over 
$\mathbb{C}$ and $\mathfrak{g}$ be the Lie algebra of $G$. If $H$ is an algebraic variety acting on $X$, we denote by $D^b_H(X,\Bbbk)$ or $D^b_H(X)$, the derived category of $H$-equivariant constructible sheaves, which is defined in \cite{BL}, and  $\p_H(X,\Bbbk)$, its full subcategory of $H$-equivariant perverse $\Bbbk$-sheaves. The constant sheaf on $X$ with value $\Bbbk$ is denoted by ${\underline{\Bbbk}}_X$ or more simply 
$\underline{\Bbbk}$.

We  fix a cocharacter map, $\chi:\mathbb{C}^\times \to G$ and define, \[G_0=\{g\in G| g\chi(t)=\chi(t)g, \forall{t}\in \mathbb{C}^\times\}.\]
 For $n\in \mathbb{Z}$, define, \[\mathfrak{g}_n=\{x\in \mathfrak{g}| Ad(\chi(t))x=t^nx, \forall{t}\in \mb{C}^\times\}.\] This defines a grading on $\mathfrak{g}$,
 \[\mathfrak{g}=\bigoplus_{n\in \mathbb{Z}}\mathfrak{g}_n.\]Clearly, $\mathfrak{g}_0=\lie(G_0)$ and
 $G_0$ acts on $\mathfrak{g}_n$. We have the following lemma  from \cite[pp. ~158]{Lu},

\begin{lemma}\label{finite}
	For $n\neq0$, $G_0$ acts on $\mathfrak{g}_n$ with only finitely many orbits.	
\end{lemma}

Recall that $\mf{sl}_2$ is the Lie algebra of $SL_2$ generated by,
 \[e=\begin{pmatrix}
0&1\\0&0\\	
\end{pmatrix}, h=\begin{pmatrix}
1&0\\0&-1\\	
\end{pmatrix}, f=\begin{pmatrix}
0&0\\1&0\\	
\end{pmatrix}.
\]
 Let $J_n=\{\phi:\mf{sl}_2 \to \mf{g}|\hspace{1mm}\phi(e)\in \mf{g}_n, \phi(f)\in \mf{g}_{-n},\phi(h)\in \mf{g}_0\}$. We have a action of $G_0$ on $J_n$ by $(g,\phi)\to Ad(g)\circ \phi$. It is easy to check that this action is well-defined. 
 \begin{theorem}\label{th1}
 	The map from the set of $G_0$-orbits on $J_n$ to the set of $G_0$-orbits on $\mf{g}_n$ defined by $\phi \to \phi(e)$ is a bijection.
 	
 \end{theorem}
 \begin{proof}
 The proof follows from \cite[\text{Prop} ~3.3]{Lu}.	
 \end{proof}

 \subsection{The set $\mathscr{I}(G,\Bbbk)$ and $\mathscr{I}(\mathfrak{g}_n,\Bbbk)$}\label{subsec2.1} Let $\mathscr{N}_G$ be the nilpotent cone of $G$. Recall that $G$ acts on $\mathscr{N}_G$ and has finitely many orbits.
 
 The set $\mathscr{I}(G,\Bbbk)$ is the set of pairs $(C,\mathcal{E})$ satisfying the condition that $C \subset \mathscr{N}_G$ is a nilpotent $G$-orbit in $\mathfrak{g}$ and $\mathcal{E}$ is an irreducible $G$-equivariant $\Bbbk$-local system on $C$(up to isomorphism). $G$-equivariant local systems on $C$ are in one-to-one correspondence with the irreducible representations of the component group $A_G(x):=G_x/G_x^o$ on $\Bbbk$-vector spaces, where $x$ is in $C$. Hence, it follows that the set $\mathscr{I}(G,\Bbbk)$ is finite. Sometimes when there is no confusion about the field of coefficients then we will just use $\mathscr{I}(G)$.
 
 Let $\mathscr{I}(\mathfrak{g}_n,\Bbbk)$ or $\mathscr{I}(\mathfrak{g}_n)$ be the set of all pairs $(\mathcal{O},\mathcal{L})$ where $\mathcal{O}$ is a  $G_0$-orbit in $\mathfrak{g}_n$ and $\mathcal{L}$ is an irreducible, $G_0$-equivariant $\Bbbk$-local system on $\mathcal{O}$(upto isomorphism). For fixed $\mathcal{O}$, $G_0$-equivariant local systems on $\mathcal{O}$ are in one to one correspondence with the irreducible representation of $A_{G_0}(x):=(G_0)_x/(G_0)_x^o$ for $x\in \mathcal{O}$. Hence by Lemma \ref{finite}, $\mathscr{I}(\mathfrak{g}_n)$ is finite.

Recall $G_0$ acts on $\mf{g}_n$ by the adjoint action. 
Now we have $\mb{C}^\times \times G_0$ action on $\mf{g}_n$ by $(t,g)\to t^{-n}Ad(g)$.

\begin{lemma}\label{lm3.9}
The $\mb{C}^\times \times G_0$-orbits and $G_0$-orbits coincide and each $G_0$- equivariant local system is also $\mb{C}^\times \times G_0$- equivariant and hence $\mb{C}^\times$-equivariant.
\end{lemma}
\begin{proof}
Since there are finitely many $G_0$-orbits in $\mf{g}_n$,  we can choose a $\mb{C}^\times$-line $L$ in $\mf{g}_n$ and $G_0$-orbit $ \mc{O}$ so that $\mc{O}\cap L$ is dense in $L$. Now we can choose $x\in \mc{O}\cap L$. Therefore $L=\mb{C}^\times \cdot x$. Let $y\in \mb{C}^\times\cdot x-\mc{O}$ and let $\mc{O}'$ be the $G_0$-orbit of $y$. As $y$ is in the closure of ${\mc{O}\cap \mb{C}^\times \cdot x}$, which is  a  subset of  $\bar{\mc{O}}$.  Therefore,  $\mc{O}'\subset \bar{\mc{O}}$. Hence $\dim \mc{O}'< \dim \mc{O}$. Now as $\mb{C}^\times$ action commutes with $G_0$ action, so $G^x_0=G^y_0$. Hence $\dim \mc{O}'=\dim \mc{O}$, which is a contradiction. So we can conclude that $\mb{C}^\times\cdot x - \mc{O}$ is empty.

For the second part, it is quite easy to show that 

\[(\mb{C}^\times \times G_0)^x/(\mb{C}^\times \times G_0)^{x,\circ}\cong G_0^x/(G_0^x)^\circ.\]
Hence, $\loc_{,G_0}(\mf{g}_n)\cong \loc_{, \mb{C}^\times \times G_0}(\mf{g}_n)$. As $\mb{C}^\times$ is sitting inside $\mb{C}^\times \times G_0$ ,  any $\mb{C}^\times \times G_0$-equivariant sheaf is $\mb{C}^\times$-equivariant. Hence we can conclude the $G_0$-equivariant local system is $\mb{C}^\times \times G_0$-equivariant and hence $
\mb{C}^\times$-equivariant.
\end{proof}

 \subsection{Cuspidal pairs}\label{cuspidal}
  The simple objects in $\p_G(\mathscr{N}_G,\Bbbk)$ are of the form $\mathcal{IC}(C,\mathcal{E})$, where $(C,\mc{E})\in \ms{I}(G)$. Let $P$ be a parabolic subgroup of $G$ with unipotent radical $U_P$ and let $L\subset P$ be a Levi factor of $P$. One can identify $L$ with $P/U_P$ through the natural morphism, $L\xhookrightarrow{}P \twoheadrightarrow{} P/U_P$. We consider a diagram,
  \[\mathscr{N}_L\xleftarrow{\pi_P}\mathscr{N}_L+\mathfrak{u}_P\xrightarrow{e_P}G\times^{P}(\mathscr{N}_L+\mathfrak{u}_P)\xrightarrow{\mu_P}\mathscr{N}_G
  \] where $\mathfrak{u}_P=\lie(U_P)$, $\pi_P, e_P$ are the obvious maps and $\mu_P(g,x)=Ad(g)x$. Let \[i_P=\mu_P\circ e_P:\mc{N}_L+\mf{u}_P\to \mc{N}_G\]The parabolic restriction functor denoted by,
  \[\Res^G_P: D^b_G(\mathscr{N}_G,\Bbbk) \to D^b_L(\mathscr{N}_L,\Bbbk)
  \]is defined by $\Res^G_P(\mathcal{F})={\pi_P}_!e_P^*\mu_P^*\f^G_L(\mathcal{F})={\pi_P}_!i_P^*\f^{G}_L(\mc{F})$. Here \[\f^G_L: D^b_G(\mc{N}_G,\Bbbk) \to D^b_L(\mc{N}_G,\Bbbk)\] is the forgetful functor.
   The parabolic induction comes from the same diagram above.
   \[\Ind^G_P:  D^b_L(\mathscr{N}_L,\Bbbk) \to D^b_G(\mathscr{N}_G,\Bbbk)  \]
and is defined by,
   $\Ind^G_P(\mc{F}):={\mu_P}_!(e_P^*\f^G_P)^{-1}\pi_P^*(\mc{F})$. 
 Here again $\f^G_P$ denotes the forgetful functor and,  $e_P^*\f^G_P:D^b_G( G\times^P(\mc{N}_L+\mf{u}_P))\to D^b_P(\mc{N}_L+\mf{u}_P)$ is the induction equivalence map.
\begin{definition}
\begin{enumerate}
	\item A simple object $\mathcal{F}$ in $\p_G(\mathscr{N}_G,\Bbbk)$ is called cuspidal if $\Res^G_P(\mc{F})=0$, for any proper parabolic $P$ and  Levi factor $L\subset P$.
	\item A pair $(C,\mathcal{E})\in \mathscr{I}(G)$, is called cuspidal if the corresponding simple perverse sheaf $\mathcal{IC}(C,\mathcal{E})$ is cuspidal.

\end{enumerate}	
\end{definition} 
\begin{remark}
		Notice that the set of cuspidal pairs  depends on the characteristic $l$ of the field of coefficients $\Bbbk$; so we sometime call it $l$-cuspidal. We will denote the subset of cuspidal pairs in $\mathscr{I}(G)$ or $\mathscr{I}(G,\Bbbk)$ by $\mathscr{I}(G)^{\cu}$ or $\mathscr{I}(G,\Bbbk)^{\cu}$.
\end{remark}

\begin{remark}\label{rk1.2.1}
From \cite[\text{Remark} ~2.3(1)]{AJHR2}, if $\mathcal{IC}(C,\mathcal{E})$ is cuspidal then so is $\mathbb{D}\mathcal{IC}(C,\mathcal{E})=\mathcal{IC}(C,\mathcal{E}^\lor)$, where $\mathbb{D}$ is the Verdier duality functor and $\mathcal{E}^\lor$ is the dual local system of $\mathcal{E}$.  	
\end{remark}

\subsection{Modular reduction} Let $\mathbb{K}$ be a finite extension of $\mathbb{Q}_l$ with ring of integers $\mathbb{O}$ and residue field $\Bbbk$. Then $(\mathbb{K},\mathbb{O},\Bbbk)$ constitutes an $l$-modular system  and we can talk about the modular reduction map. Let $E\in (\mb{K},\mb{O},\Bbbk)$, and $K_{G_0}(\mf{g}_n,E)$ be the Grothendieck group of $D^b_{G_0}(\mf{g}_n,E)$. Then the modular reduction map,
\[d:K_{G_0}(\mf{g}_n,\mb{K})\to K_{G_0}(\mf{g}_n,\Bbbk)\] is defined by $d[\mc{IC}(\mc{O},\mc{L})]=[\Bbbk \otimes^L_\mc{O} \mc{IC}(\mc{O},\mc{L}_\mb{O})]$, where $\mc{L}_\mb{O}$ is a torsion-free part $\mb{O}$-local system. In the same way we can define the modular reduction on the nilpotent cone,
\[K_G(\mc{N}_G,\mb{K}) \to K_G(\mc{N}_G,\Bbbk).\]If the characteristic $l$ of $\Bbbk$ is rather good for $G$(see Definition \ref{rg}), this modular reduction induces a bijection by \cite{AJHR},
\[\ms{I}(G,\mb{K}) \xrightarrow{\cong} \ms{I}(G,\Bbbk).\]
We will discuss the modular reduction in more detail in \ref{sec5.3}. The pair 
$(C,\mathcal{E})\in \mathscr{I}(G,\Bbbk)$ will be called $0$-cuspidal if it is in the image of some cuspidal pair under $d$, and we will denote the set of $0$-cuspidal pairs by $\mathscr{I}(G,\Bbbk)^{0-\cu}$ or $\mathscr{I}(G)^{0-\cu}$. 
\begin{definition}
 $(\mathcal{O},\mathcal{L})\in \mathscr{I}(\mathfrak{g}_n,\Bbbk)$ will be called  cuspidal if there exists a pair $(C,\mc{E})\in \ms{I}(G)^{0-\cu}$, such that $C\cap \mf{g}_n=\mc{O}$ and $\mc{L}=\mc{E}|_{\mc{O}}$.
  We will denote the set of all cuspidal pairs on $\mf{g}_n$ by, 
$\mathscr{I}(\mathfrak{g}_n)^{\cu}$.
\end{definition}
\begin{remark}
Notice that the definition of cuspidal on $\mf{g}_n$ is not  coming from the restriction functor as for the nilpotent cone and there is no $l$ version in definition of cuspidal pairs on $\mf{g}_n$. 
\end{remark}

\subsection{Cleanness}\label{subsec2.3}A pair $(C,\mathcal{E})\in \mathscr{I}(G)$ is called $l$-clean if the corresponding $\mathcal{IC}(C,\mathcal{E})$ has vanishing stalks on $\bar{C}-C$. Similarly, a pair $(\mathcal{O},\mathcal{L})\in \mathscr{I}(\mathfrak{g}_n)$ is called $l$-clean if the corresponding $\mathcal{IC}(\mathcal{O},\mathcal{L})$ has vanishing stalks on $\bar{\mathcal{O}}-\mathcal{O}$. 

\begin{definition}\label{rg}
	A prime number $l$ is said to be a rather good prime for a group $G$, if it is a good prime for $G$ and does not divide $|Z(G)/Z(G)^\circ|$.
\end{definition}
By \cite[Lemma ~2.1]{AJHR}, a prime is rather good if and only if $l$ does not divide $|A_G(x)|$ for any $x\in \mc{N}_G$.
The following is a part of a series of (unpublished) conjectures by C. Mautner.
\begin{conjecture}\label{con}{(Mautner's cleanness conjecture)} If the characteristic $l$ of $\Bbbk$ is a rather good prime for $G$, then every $0$-cuspidal pair $(C,\mathcal{E}) \in \mathscr{I}(G)$ is $l$-clean. 
	
\end{conjecture}

\begin{remark}
This conjecture has been already proved when the characteristic $l$ does not divide the order of the Weyl group of $G$, or if every irreducible factor of the root system of $G$ is either of type $A,B_4,C_3,D_5$ , or of exceptional type \cite{AJHR}.	
\end{remark}

\begin{remark}
The cleanness conjecture is not true if we replace $0$-cuspidal by $l$-cuspidal. A counter example is when $G=GL(2)$, and $l=2$ then the unique $2$-cuspidal pair $(O_{(2)},\underline{\Bbbk})$ is not $2$-clean. The proof is explained in \cite[Remark ~2.5]{AJHR2}.	 
\end{remark}

In this paper we  assume this conjecture is true.


  \subsection{Parity sheaves}\label{parity}Let $H$ be a linear algebraic group and  $X$ be a $H$-variety. We fix a stratification\[
  X=\coprod_{\lambda\in\Lambda}X_{\lambda}
  \]of $X$ into smooth connected locally closed ($H$-stable) subsets. For each $\lambda\in \Lambda$, $i_{\lambda}:X_{\lambda}\hookrightarrow X$ denotes the inclusion map and let $d_{\lambda}$ be the dimension of $X_{\lambda}$. For each $\lambda\in \Lambda$, let $\loc_{,H}(X_{\lambda},\Bbbk)$ or $\loc_{,H}(X_{\lambda})$ denote the category of $H$-equivariant $\Bbbk$-local systems of finite rank on $X_{\lambda}$.
 
 According to \cite{parity}, to talk about parity sheaves on $X$ we need the condition below, 
 \begin{equation}\label{vanish}
 	 \h^*_H(\mc{L})=0 \text{   for odd degrees,}
 \end{equation}
for any local system $\mc{L}\in \loc_{,H}(X_{\lambda})$ and for any $\lambda\in \Lambda$. 
\begin{definition}
\begin{enumerate}
\item A complex $\mathcal{F}\in D^b_H(X)$ is called $*$-even if	for each $\lambda\in \Lambda$ and $n\in \mathbb{Z}$, $\h^n(i_{\lambda}^*\mathcal{F})$ belongs to $\loc_{,H}(X_{\lambda})$ and vanishes for $n$ odd. A complex $\mathcal{F}\in D^b_H(X)$ is called $*$-odd if	for each $\lambda\in \Lambda$ and $n\in \mathbb{Z}$, $\h^n(i_{\lambda}^*\mathcal{F})$ belongs to $\loc_{,H}(X_{\lambda})$ and vanishes for $n$ even. Similarly, we can define $!$-even and $!$-odd complexes. 
  	
  	\item
  	A complex $\mathcal{F}$ is called even if it is both $*$-even and $!$-even.
  	 	A complex $\mathcal{F}$ is called odd if it is both $*$-odd and $!$-odd.
  	\item
  	A complex $\mathcal{F}$ is called parity if it splits as the direct sum of an even complex and an odd complex.
  \end{enumerate}

  \end{definition}
  Like $\mathcal{IC}$ sheaves, the definition of parity sheaves comes from a theorem. The following theorem  requires the assumption (\ref{vanish}) on of $\loc_{, H}(X_\lambda)$.
  
  \begin{theorem}\label{th2.3}
  Let $\mathcal{F}$ be an indecomposible parity complex. Then
  \begin{enumerate}
  \item The support of 	$\mathcal{F}$ is irreducible and hence of the form $\bar{X_{\lambda}}$, for some $\lambda \in \Lambda$.
  \item $\mathcal{F}|_{X_\lambda}$ is isomorphic to $\mathcal{L}[m]$ for some indecomposible object $\mathcal{L}$ in $\loc_{,H}(X_{\lambda})$ and some integer $m$.
  \item Any indecomposable parity complex supported on $\bar{X}_\lambda$ and extending $\mathcal{L}[m]$ is isomorphic to $\mathcal{F}$.
  \end{enumerate}
 	
  \end{theorem}
The proof of the theorem is given in  \cite[ ~2.12]{parity}.

\begin{definition}\label{parity}
A parity sheaf is an indecomposable parity complex with support $\bar{X}_\lambda$	and extending $\mathcal{L}[m]$ for some indecomposable $\mathcal{L}\in \loc_{,H}(X_\lambda)$ and for some $m\in \mb{Z}$. When such a complex exists we  denote it by $\mathcal{E}(X_\lambda,\mathcal{L})$ or $\mathcal{E}(\lambda,\mathcal{L})$ and this  has the property that $\mc{E}(\lambda,\mc{L})|_{X_\lambda}=\mc{L}[\dim X_\lambda]$. This is the unique parity sheaf associated with $(\lambda,\mathcal{L})$ up to shift.
\end{definition}
\begin{remark}
\begin{enumerate}
\item If $\mathcal{L}$ is not indecomposable then $\mathcal{E}(\lambda,\mathcal{L})$ denotes the direct sum of parity complexes coming from the direct summand of $\mathcal{L}$.
\item If $\mathcal{L}=\underline{\Bbbk}_{X_\lambda}$, then we may write $\mathcal{E}(\lambda, \mathcal{L})$ as $\mathcal{E}(\lambda)$.	
\end{enumerate}
	
\end{remark}

\subsection{Torsion primes and pretty good primes}\label{subsec1.6}Let $G$ be a reductive group with the root datum $(\bf{X},\Phi,\bf{Y},\Phi^\lor)$. A reductive subgroup of $G$ is called regular if it contains a maximal torus. If the  group $G$ is complex reductive  then all the regular reductive subgroups are in bijection with $\mathbb{Z}$-closed subsystems of $\Phi$, that is $\Phi_1\subset \Phi$.
\begin{definition}
	A prime $p$ is called a torsion prime for $G$ if for some regular reductive subgroup $H$ of $G$, $\pi_1(H)$ has $p$-torsion.
\end{definition}
\begin{definition}
A 	 prime $p$ is called pretty good for $G$ if for all subsets $\Phi_1\subset \Phi$, $\bf{X}/{\mathbb{Z}\Phi_1}$ and $\bf{Y}/{\mathbb{Z}\Phi_1^\lor}$ have no $p$-torsion.
\end{definition}
The properties of reductive groups for which $p$ is a pretty good prime have been discussed in \cite[\text{Remark} ~5.4]{Her}. From those properties and using the tables of centralisers from \cite{Car}, we have the following lemma.
\begin{lemma}\label{lm1.4}
A prime $p$ is pretty good for $G$ if and only if for all $x\in \mc{N}$, $p$ is not a torsion prime for $C_x$, where $C_x $ is the maximal reductive quotient of $(G^x)^\circ$, and the order of $A_G(x)$ is invertible in $\Bbbk$.	
\end{lemma}

This is the right time to make some assumption on the characteristic of $\Bbbk$.
  
  \begin{assume}\label{char}\begin{enumerate}\item  The characteristic $l$ of $\Bbbk$ is a pretty good prime for $G$.\item The field $\Bbbk$ is big enough for $G$;
i.e, for every Levi subgroup $L$ of $G$ and pair $(C_L,\mathcal{E}_L)\in \mathscr{I}(L)$, the irreducible $L$-equivariant $\Bbbk$-local system $\mathcal{E}_L $ is absolutely irreducible.  	
  \end{enumerate}
	
  \end{assume}
\begin{remark}
	Note that pretty good implies that   $|A_G(x)|$ is invertible in $\Bbbk$. Hence pretty good implies rather good. So if  conjecture \ref{con} holds then it in particular holds for  pretty good primes.
 \end{remark}
 
 \begin{remark}\label{rk2.4.2}
 	From \cite[Lemma ~2.2(1)]{AJHR},  if a prime $l$ is rather good for $G$ then it is rather good for all the Levi subgroups. Hence $|A_L(x)|$ is still invertible in $\Bbbk$ for any Levi subgroup $L\subset G$.
 \end{remark}
Recall that nilpotent orbits are even dimensional\cite[~1.4]{CM}.
As a direct consequence of the above lemma, we have the next theorem.

\begin{theorem}\label{th3}
Let $C$ be a nilpotent orbit in $\mf{g}$ and $\mc{L}\in \loc_{,G}(C,\Bbbk)$, then $\h^*_G(\mc{L})$ vanishes in odd degrees.	
\end{theorem}

\begin{proof}
The proof is given in  \cite[\text{Lemma} ~4.17]{parity}	
\end{proof}

  \begin{theorem}\label{th1.6}
Let $\mc{O}$ be a  $G_0$-orbit in $\mf{g}_n$ and $\mc{L}\in \loc_{,G_0}(\mc{O},\Bbbk)$, then $\h^*_{G_0}(\mc{L})$ vanishes in odd degrees.	\end{theorem}

The proof of this theorem will be given in section 6.

Now, talking about parity sheaves makes sense both in $G$- and $G_0$-equivariant settings as we know (\ref{vanish}) is true for $\mc{N}_G$ and $\mf{g}_n$.

\begin{conjecture}\label{2.6(c)}
Let $P$ be a parabolic subgroup of $G$ and $L$ be its Levi subgroup. For a pair $(C,\mc{E})\in \ms{I}(L)^{0-\cu}$, $\Ind^G_P\mathcal{IC}(C,\mathcal{E})$ is a parity complex.
\end{conjecture}
 In characteristic $0$, the proof follows from the decomposition theorem and \cite[~24.8]{Lu2}.	
In positive characteristic, the result is still unknown. Throughout this paper we will assume this result is true. In the last section we will give some example where the conjecture holds.

\section{Lemmas on varieties with $\mb{C}^\times$-action}\label{sec3}

Let $X$ be a variety defined over $\mathbb{C}$ and $G$ be a connected linear algebraic group which acts on $X$. In this section all the sheaf coefficients will be considered over $\Bbbk$ whose characteristic satisfies assumption \ref{char}.
\begin{lemma}\label{lm5}
 If $\mathcal{F} \in D^b_G(X)$ and $\overline{\mathcal{F}}= \f(\mathcal{F}) \in D^b_c(X)$, then $\h^*_c(\overline{\mathcal{F}})=0$ implies $\h^*_{G,c}(\mathcal{F})=0$.
 Equivalently, $R\Gamma_c (\overline{\mathcal{F}})=0 $ implies $R\Gamma_c(\mathcal{F})=0$
 \end{lemma}
 \begin{proof}
 	$R\Gamma_c(\overline{\mathcal{F}})= For(R\Gamma_c\mathcal{F}) $. Now the argument follows from the following statement, for $X$  as defined above and 
 $\mathcal{M} \in D^b_G(X)$, if $\f(\mathcal{M})=0$ then $\mathcal{M}=0$. The proof of this statement easily follows using induction on the number of nonzero cohomology and truncation. 

 \end{proof}
 
\begin{lemma}\label{lm6}
	   Let $V$ be a finite-dimensional vector space with a nontrivial linear $\mb{C}^\times$-action on it. Then there exists a nonzero vector with a stabilizer of minimum size and  there exists a nonzero vector with a  stabilizer of maximum size among those with finite stabilizers. 
\end{lemma}

\begin{proof}	   As $\mb{C}^\times$ acts on a vector space $V$,  we have a grading on $V$ given by,
   \[
   V=\bigoplus_{\lambda\in \mb{Z} }V_\lambda,
   \] where $V_\lambda$'s are eigenspaces of the $\mb{C}^\times $-action.
    Note that  stabilizers are subgroups of $\mb{C}^\times$. If they are finite then they are cyclic and  of the form $\mb{Z}/n\mb{Z}$. So it is obvious that there exists an element of minimum-sized stabilizer.
    Now let $v\in V$ and $v\notin V_0$. Then we can write $v$ as:
$v=\oplus v_{\lambda}$, where $v_{\lambda}\in V_\lambda$. By definition $t.v_\lambda=t^\lambda v_\lambda$. Hence $t\in \stab(v_\lambda)$ if and only if $t$ is a $\lambda$-root of unity. In other words, $|\stab(v_\lambda)|=|\lambda|$ for $\lambda \ne 0$. Now if $t$ is in $\stab(v)$, then $t$ must stabilize all the $v_\lambda$'s.  Hence $t$ must be $c$-th root of unity where $c$ divides $\lambda$ for all $\lambda\ne 0$.  Hence $\max{|\stab(v)|}=\max\{|\lambda| | V_{\lambda} \ne 0\}$.
\end{proof}
 \begin{lemma}\label{lm7}
 Let $\mathbb{C}^\times$ acts on $Y$, a variety over $\mathbb{C}$, with finite stabilizers. Assume that all the stabilizers have order not divisible by $l$, where $l$ is the characteristic of $\Bbbk$. Then for any object $\mathcal{F} \in D^b_{\mathbb{C}^\times}(Y,\Bbbk)$, $\dim(\h^*_{\mathbb{C}^\times}(Y,\mathcal{F}))< \infty$.
\end{lemma}
\begin{proof}
   For a general variety $Y$,  as all the subgroups of $\mathbb{C}^\times$ are finite,  there exists a stabilizer of minimum size, say $n$. Let $U= \{y\in Y,  |\stab(y)|= n\}$.  By the Sumihiro embedding theorem  \cite{Su}, we can cover $Y$ by $\mathbb{C}^\times$  invariant subvarieties, each of which  is equivariantly isomorphic to a $\mathbb{C}^\times$-invariant closed subvariety of $\mathbb{A}^N$ for some $N$, on which $\mathbb{C}^\times$ acts linearly. 
Now this action is not trivial as we have already  assumed that $\mb{C}^\times$ acts nontrivially on $Y$ with finite stabilizers. Hence, by Lemma \ref{lm6}, these stabilizers will have maximum size. Now the claim is that $U$ is open. Let $Z=U^c=\{y\in Y| |\stab(y)|>n\}$. As the maximum sized stabilizer exists,  we can choose a finite subgroup $M$ of $\mathbb{C}^\times$ which contains all the stabilizers. Let $m\in M$ which is not in the minimum sized stabilizers. Then $Z_m=\{y\in Y |m\in\stab(y)\}$ is closed. The collection $\{Z_m\}$ is finite as $m$ is coming from a finite subgroup $M$. Also $Z=\cup_m Z_m$, finite union of closed sets, hence is closed. So $U$ is open.
 
 So we have the open and closed embeddings,  
 \[
 Z\xhookrightarrow{i}Y\xhookleftarrow{j}U\] which give us the distinguished triangle,
 \begin{equation}\label{5}
i_*i^!\mathcal{F}\to\mathcal{F}\to j_*j^*\mathcal{F}\to\hspace{2mm}.
 \end{equation}

 Now we are at a place to use Noetherian induction on $Y$. The theorem is true for the empty set. So we can  assume that it is true for all the proper closed subvarieties of $Y$, particularly for $Z$.
Now for $u\in U$, let $K=\stab(u)=\mathbb{Z}/n\mathbb{Z}$. Let $H=\mathbb{C}^\times/K$, then $H$ acts freely on $U$. 

Let $\mc{F}\in D^b_{\mb{C}^\times}(U,\Bbbk)$. The goal is to show $\dim \h^*_{\mb{C}^\times}(\mc{F})<\infty$.
According to \cite[ Sec ~6]{BL}, if $G'$ and $G$ are two topological groups acting on two varieties $X$ and $Y$ respectively, and $\phi:G'\to G$ be a homomorphism of topological groups with $f:X\to Y$, a $\phi$-equivariant map, then there exist two functors,
\[Q^*_f:D^+_G(Y)\to D^+_{G'}(X)\] and
\[Q_{f*}:D^+_{G'}(X)\to D^+_G(Y).\]Here $Q^*_f$ and $Q_{f*}$ are adjoint to each other.
In our case, we take both $X$ and $Y$ to be $U$ and $G'=\mb{C}^\times$,  $G=H$ and we take $f$ to be the identity. So we have,
\[Q_{id*}: D^+_{\mb{C}^\times}(U) \to D^+_H(U).\]
Therefore if $\mc{F}\in D^b_{\mb{C}^\times}(U)$, then $\Hom^*_{D^+_H(U)}(\underline{\Bbbk}, Q_{id*}\mc{F})\cong \Hom^*_{D^+_{\mb{C}^\times}(U)}(\underline{\Bbbk},\mc{F})$.
Now if we can show $Q_{id*}\mc{F}$ is in $D^b_H(U)$, then as we already know $D^b_H(U)\cong D^b_c(U/H)$(non-equivariant) as $H$ acts freely on $U$, so $\dim (\Hom^*_{D^b_c(U/H)}(\underline{\Bbbk},Q_{id*}\mc{F}) )< \infty$.
The next following fact is from \cite[~ 7.3]{BL}.
For $\mc{F} \in D^b_{\mb{C}^\times}(U) \subset D^+_{\mb{C}^\times}(U)$, 
we have the commutative diagram below,
\[
\begin{tikzcd}
D^+_{\mb{C}^\times}(U) \arrow[r, "Q_{id*}"] \arrow[d, "\f^{\mb{C}^\times}_K"] & D^+_{\mb{C}^\times/K}(U) \arrow[d, "\f^{\mb{C}^\times/K}"] \\
D^+_K(U) \arrow[r, "Q_{id*}"]                     & D^+(U)                            
\end{tikzcd}
\] 
Therefore we have $\f^{\mb{C}^\times/K}Q_{id*}\mc{F}\cong Q_{id*}\f^{\mb{C}^\times}_K \mc{F}$. Let $\mc{G}=Q_{id*}\mc{F}$. So $\mc{G}$ is bounded if and only if $\f^H\mc{G}$ is bounded.  As $\mc{F}$ is from bounded derived category then so is $\f^{\mb{C}^\times}_K\mc{F}$. Therefore to show $\f^H\mc{G}$ is bounded it is enough to show,
\[Q_{id*}:D^b_K(U) \to D^+(U)\] takes values in $D^b(U)$. As $K$ is finite hence discrete, so by \cite[Cor ~8.4.2]{BL}, $D^b_K(U)\cong D^bSh_K(U)$. But we know $ Sh_K(U,\Bbbk) \cong Sh(U,\Bbbk[K])$, where $\Bbbk[K]$  is a commutative semisimple ring as $l\nmid |K|$. By the same corollary, $Q_{id*}$ corresponds to a exact functor. So we can conclude $Q_{id*}$ takes $D^b_K(U)$ to $D^b(U)$.
 Now coming back to our actual proof, if we apply ${a_X}_*$ to (\ref{5}), where $a_X :X \to \{pt\}$, we get,
\[
\h^*_{\mb{C}^\times}(i^!\mc{F})\to \h^*_{\mb{C}^\times}(\mc{F})\to \h^*_{\mb{C}^\times}(j^*\mc{F})\to
\]As $\dim(\h^*_{\mb{C}^\times}(i^!\mc{F})<\infty$ by induction hypothesis and $\dim(\h^*_{\mb{C}^\times}(j^*\mc{F}))<\infty$ by the above result. Hence $\dim(\h^*_{\mb{C}^\times}(\mc{F}))<\infty$.
\end{proof}
\begin{lemma}\label{lm8}
	Let $X$ be a variety over $\mb{C}$ and $H$ be a connected linear algebraic group acting trivially on $X$. Then for $\mc{F}\in D^b_H(X,\Bbbk)$,
	\[
	\h^*_{H,c}(X,\mc{F})\cong \h^*_c(X,\mc{F})\otimes\h^*_{H}(pt,\Bbbk).
	\]
\end{lemma}
\begin{proof}
Let $U_n\to \{pt\}$ be the $n$-acyclic resolution for $\{pt\}$. Then $U_n\times X\to X$ is the $n$-acyclic resolution of $X$. By \cite[ ~2.1]{BL}, if  $\mc{F}\in D^b_H(X,\Bbbk)$  this implies $\mc{F}\boxtimes\underline{\Bbbk}_{U_n/H} \in D^b_c(X \times U_n/H,\Bbbk)$ such that for $i<n$,
\begin{align*}
\h^i_{H,c}(X,\mc{F})&\cong \Hom_{D^b_H(pt,\Bbbk)}(\underline{\Bbbk}_{pt},{a_X}_!\mc{F}[i])\\
& \cong \Hom_{D^b_c(U_n/{H}\times X)}(\underline{\Bbbk},\underline{\Bbbk}\boxtimes {a_X}_!\mc{F}[i])\\
&\cong \h^i(\rg(\underline{\Bbbk}_{U_n/H}\boxtimes {a_X}_!\mc{F}).
\end{align*}
As for constructible sheaves all sheaf functors commute with $\boxtimes$.    Therefore we have for $i<n$,
\begin{align*}
\h^i(\rg(\underline{\Bbbk}_{U_n/H}\boxtimes \mc{F}))
&\cong \bigoplus_{j+k=i} \h^j(\rg(\underline{\Bbbk}_{U_n/H}))\otimes \h^k(\rg(\mc{F}))\\
&\cong \bigoplus_{j+k=i}\h^j_H(pt,\Bbbk)\otimes \h^k(X,\mc{F}).
\end{align*}
So we are done.

\end{proof}

\begin{lemma}\label{lm9}
Let $Y$ be an algebraic variety over $\mb{C}$ and $Y_0$ be the fixed point set of this action. Assume that $\mb{C}^\times$ acts on $Y-Y_0$ with finite stabilizers and all the stabilizers of $Y-Y_0$ have order not divisible by $l$. Let $\mc{F}\in D^b_{\mb{C}^\times}(Y,\Bbbk)$. 
If          $\h^j_c(Y,\mc{F})=0$ for all $j$, then $\h^j_c(Y_0,\mc{F})=0$ for all $j$.
	
\end{lemma}
\begin{proof}
	Let $\mathcal{F}\in D^b_{\mathbb{C}^\times}(Y)$ and $\h^j_c(Y,\mathcal{F})=0$. By Lemma \ref{lm5}, $\h^j_{\mathbb{C}^\times,c}(Y,\mathcal{F})=0$. Let $Y_1=Y-Y_0$, then we have the open and closed embeddings,
	\[Y_0\xhookrightarrow{i} Y \xhookleftarrow{j} Y_1. \]
	This gives us the long exact sequence of $\mathbb{C}^\times$ equivariant cohomology
	\[
	j_!j^*\mathcal{F}\to\mathcal{F}\to i_*i^*\mathcal{F}\to\hspace{2mm}.
	\]
	From that we get,
  $\h^i_{\mathbb{C}^\times,c}(Y_0,\mathcal{F})=\h^{i+1}_{\mathbb{C}^\times,c}(Y_1,\mathcal{F})$, for all $i$. By Lemma \ref{lm7},  $\dim(\h^*_{\mathbb{C}^\times,c}(Y_1,\mathcal{F}))$ is finite. Therefore we can conclude that  $\dim(\h^*_{\mathbb{C}^\times,c}(Y_0,\mathcal{F}))$ is also finite.  By Lemma \ref{lm8},
\begin{equation}\label{99}
\h^*_{\mathbb{C}^\times,c}(Y_0,\mathcal{F})\cong \h^*_{c}(Y_0,\mathcal{F})\otimes  \h^*_{\mathbb{C}^\times}({pt},\Bbbk).
\end{equation}
  Recall, $
  \h^*_{\mathbb{C}^\times}({pt},\Bbbk) \cong Sym(\Bbbk)$, which  is infinite dimensional.
  In equation (\ref{99}), if $\h^*_c(Y_0,\mc{F})\ne 0$ then LHS is finite dimensional and RHS is infinite dimensional,  a contradiction. So $\h^*_c(Y_0,\mc{F})=0$.
\end{proof}

\begin{lemma}\label{lm10} 	
Let $\mathcal{G}\in D^b_{\mathbb{C}^\times}(pt,\Bbbk)$.
\begin{enumerate}
	\item If  $\h^i(For^{\mathbb{C}^\times}(\mathcal{G}))=0$, for all $i$   odd, Then $\Hom({\underline{\Bbbk}}_{pt}, \mathcal{G}[i])=0$, for all $i$  odd.
	\item If  $\Hom^*(\underline{\Bbbk},\mathcal{G})$ is free over $\h^*_{\mathbb{C}^\times}(pt)$ and $0$ for  odd degrees, then $\h^*(For^{\mathbb{C}^\times}(\mathcal{G}))=0$   for odd degrees.	\end{enumerate}
\end{lemma}
\begin{proof}\begin{enumerate}\item 
By Lemma \ref{lm5}, $\h^i(\mathcal{G})=0$ for $i$ odd. If $\h^i(\mathcal{G})$ is nonzero for a unique $i$, then $\mathcal{G}=\underline{\Bbbk}_{pt}[i]$, and clearly the statement is true. If there is more than one nonzero cohomology, then we will use induction on the number of nonzero cohomology sheaves and we will use truncation on $\mathcal{G}$ to reduce to the case, $\mathcal{G}=\bigoplus \underline{\Bbbk}_{pt}[2m]$. Hence, $$\Hom({\underline{\Bbbk}}_{pt}, \mathcal{G}[i])= \bigoplus_m \Hom(\underline{\Bbbk}_{pt}, \underline{\Bbbk}_{pt}[i+2m]) = \bigoplus \h^{i+2m}_{\mathbb{C}^\times}(pt),$$
	which is zero when $i$ is odd.
	 \item Note $\h^*_{\mathbb{C}^\times}(\mathcal{G})$ is free over $\h^*_{\mathbb{C}^\times}(pt)$. So, using the fact that $\h^*_{\mathbb{C}^\times}(\mathcal{G})$ is $0$ in odd degrees, we can choose basis elements,  $\gamma_i\in \h^{-2n_i}_{\mb{C}^\times}(\mc{G})$ for $i=1,\dots, k$. Therefore $\gamma_i$ is a map from $ \underline{\Bbbk}_{pt}[2n_i]$ to $\mc{G}$. Hence we can define \[\gamma:\underline{\Bbbk}_{pt}[2n_1]\oplus \dots \oplus\underline{\Bbbk}_{pt}[2n_k]\to \mc{G}.\]This map induces an isomorphism in equivariant cohomology. Let $\mc{F}=Cone(\gamma)$. Now the claim is that $\mc{F}\cong 0$. If $\mc{F}\ne 0$, then let $k$ be the smallest integer such that $\h^k(\mc{F})\ne 0$. As $\h^k(\mc{F})\in \loc_{,\mb{C}^\times}(pt,\Bbbk)$ which is again equivalent  to finite-dimensional $\Bbbk$-vector spaces, so there is a nonzero map $\underline{\Bbbk}_{pt}\to \h^k(\mc{F})$. Now $\h^k(\mc{F})[-k]\cong \tau^{\le k}\mc{F}$. Hence we have a nonzero map, \[\underline{\Bbbk}_{pt}[-k]\to \tau^{\le k}\mc{F}\to \mc{F}.\] 
	In other words, $\h^k_{\mb{C}^\times}(\mc{F})\cong\Hom(\underline{\Bbbk}_{pt}[-k],\mc{F})\ne 0$, which is a contradiction. Hence $cone(\gamma)=0$ and $\gamma$ is an isomorphism. Therefore, \[\h^j(\f(\mc{G}))\cong \bigoplus_{i=1}^{k}\h^{j+2n_i}(\underline{\Bbbk}_{pt}),\] which is $0$ for  $j$ odd.
\end{enumerate}
\end{proof}
\begin{theorem}\label{th7}
Let $X$ be a $\mathbb{C}$-variety with a ${\mathbb{C}^\times}$-action on it.
Let	 $X^{\mathbb{C}^\times}$ be the fixed point set and $\mathcal{F}\in D^b_{\mb{C}^\times,c}(X,\Bbbk)$, a local system. If  $\h^a_c(X,\mathcal{F})=0$ for $a$ odd, then $\h^a_c(X^{\mathbb{C}^\times},\mathcal{F})=0$ for $a$ odd, provided characteristic $l$ of $\Bbbk$ does not divide the order of the stabilizers on $X-X^{\mb{C}^\times}$.
\end{theorem}
\begin{proof}Let $Z= X^{\mathbb{C}^\times}$ and $U=X-Z$. Let $i$, $j$ be the inclusion maps of $Z$ and $U$ respectively, and,\[
 Z\xhookrightarrow{i}X\xhookleftarrow{j}U.\]Also let $a:X\to \{pt\}$. Let $ \mathcal{G} =a_!(\mathcal{F})$. We have the distinguished triangle below
 
  \[j_!\mathcal{F}|_U\to\mathcal{F}\to i_!i^*\mathcal{F}\to\hspace{5mm}.\]
  
  We can apply $a_!$ to this, and  get,
  \[\mathcal{G}\to a_! i^*\mathcal{F} \to a_!(\mathcal{F}|_U)[1]\to \hspace{5mm}.\]
  Now, \[\h^i(For(\mathcal{G}))=\h^i_c(X,\mathcal{F})=0\text{ for $i$ odd.}\] Hence by Lemma \ref{lm10}, $H^i_{\mathbb{C}^\times}(\mathcal{G})=0$ for $i$ odd. This implies that the map, \[H^i_{\mathbb{C}^\times}(a_! i^*\mathcal{F})\to H^{i+1}_{\mathbb{C}^\times}(\mathcal{F}|_U)\] is injective for $i$ odd. 
 As the characteristic of $k$  does not divide the order of the stabilizers on $X-X^{\mb{C}^\times}$,    by lemma \ref{lm7}, $\dim( \h^*_{\mathbb{C}^\times}(\mathcal{F}|_U))< \infty$. 
  The claim is that $\h^*_{\mb{C}^\times}(a_!i^*\mathcal{F})$ is free over $\h^*_{\mathbb{C}^\times}(pt)$. $i^*\mathcal{F}\in D^b_{\mathbb{C}^\times}(Z)$.
  If $\mathcal{F}$ is a local system then $i^*\mathcal{F}$ is also a local system. 
  Let $\mathcal{E}=i^*\mc{F}$ is a local system on $Z$, where $\mathbb{C}^\times$ acts on $Z$ trivially. 

Hence by Lemma \ref{lm8},
\begin{align*}
\h^*_{\mathbb{C}^\times}(a_!\mathcal{E})\cong 
\h^*_{\mathbb{C}^\times}(pt)\otimes a_!\mathcal{E}
\end{align*}        
and this is free over $\h^*_{\mathbb{C}^\times}(pt)$. So in our context $\h^{odd}_{\mathbb{C}^\times}(a_!\mathcal{E})$ is either is 0 or infinite-dimensional. If infinite-dimensional, then it is a contradiction because it has an injective map to a finite dimensional cohomology. Hence it must be  $0$ for odd degrees, that is $\h^*_{\mathbb{C}^\times}(Z,\mathcal{F})=0$ for odd degrees. Now by Lemma \ref{lm10}(2), $\h^*_c(Z,\mathcal{F})=0$ for odd degrees.
 \end{proof}

\begin{theorem}\label{euler}
	Let $M$ be an object in $D^b_{\mb{C}^\times}(pt)$.  Assume that $\h^*_{\mb{C}^\times}(M)$ is finite-dimensional, then the Euler characteristic of $\h^*(M)$ (nonequivariant cohomology) is $0$.
\end{theorem}
\begin{proof}
From \cite[Th. ~3.7.1]{BL},
\[\f^{\mb{C}^\times} : D^b_{\mb{C}^\times}(pt) \to D^b(pt)
\] has a left adjoint $\inn_!$. Let $a:\mb{C}^\times \times pt \to pt$ be the projection on $pt$, which is the constant map in this case and $\nu: pt \to \mb{C}^\times \times pt$, the inclusion map. Here $\nu^*\f^{\mb{C}^\times}[-2]$ is the induction equivalence map. Then the formula for $\inn_!$ is $a_!\mb{D}(\nu^*\f^{\mb{C}^\times}[-2])^{-1}: D^b(pt)\to D^b_{\mb{C}^\times}(pt)$, where $\mb{D}$ is the equivariant Verdier duality. Therefore $\inn_!(\underline{\Bbbk}_{pt})=R\Gamma_c(\underline{\Bbbk}_{\mb{C}^\times}[2])$. So,
\begin{align}\label{w}
 \h^i(\inn_!\underline{\Bbbk}_{pt})
 \begin{cases}
 \cong \Bbbk \text{ for } i=0,-1\\ \cong 0 \text{ otherwise.}
 \end{cases}
 \end{align}
We have the distinguished triangle,
\[
\tau^{\le{-1}}\inn_!(\underline{\Bbbk}_{pt})\to \inn_!(\underline{\Bbbk}_{pt}) \to \tau^{\ge 0}\inn_!(\underline{\Bbbk}_{pt})\to \hspace{2mm}.
\]Using (\ref{w}) this distinguished triangle reduces to
\[
\Bbbk[1]\to \inn_!(\underline{\Bbbk}_{pt}) \to \Bbbk \to \hspace{2mm}.
\]
Note that, $\Hom(\inn_!\underline{\Bbbk}_{pt}, M) \cong \Hom(\underline{\Bbbk}_{pt}, \f^{\mb{C}^\times}(M))\cong \h^i(M)$ and $\Hom(\underline{\Bbbk}_{pt}, M)\cong \h^i_{\mb{C}^\times}(M)$.
Now we apply $\Hom(-,M)$ to the above distinguished triangle and get the long exact sequence,
\[\to \h^{i-1}_{\mb{C}^\times}(M) \to \h^i(M) \to \h^i_{\mb{C}^\times}(M)\to \dots \hspace{2mm}.
\]From the assumption on $M$, this LES have finitely many terms. Therefore,
\[\chi(\h^{i-1} _{\mb{C}^\times}(M))+\chi(\h^i _{\mb{C}^\times}(M))=\chi(\h^i(M)).
\]Here $\chi$ denotes the Euler characteristics. But  
$ \chi(\h^{i-1} _{\mb{C}^\times}(M)) =- \chi(\h^{i} _{\mb{C}^\times}(M)) $, so the left hand side is $0$. So $\chi(\h^i(M))=0$ and we are done.

\end{proof}

\section{Induction and restriction}\label{sec4}
Let $P $ be a parabolic subgroup of $G$ containing $\chi(\mathbb{C}^\times)$. Let $L$ and $U$ be a  Levi subgroup and the unipotent radical, respectively. 
 We can choose $L$ so that $\chi $ gets mapped in to $L$. Let $\mathfrak{p}, \mathfrak{l}, \mathfrak{n}$ be the Lie algebras of $P,L,U$ respectively. Then $\mathfrak{p}, \mathfrak{l}, \mathfrak{n}$ inherit grading from $\mathfrak{g}$:
\[\mathfrak{p}=\bigoplus_{n\in \mathbb{Z}}{\mathfrak{p}}_n, \mathfrak{l}=\bigoplus_{n\in \mathbb{Z}}\mathfrak{l}_n, \mathfrak{n}=\bigoplus_{n\in \mathbb{Z}}\mathfrak{n}_n,
\] where $\mathfrak{p}_n=\mathfrak{p}\cap \mathfrak{g}_n, \mathfrak{n}_n=\mathfrak{n}\cap \mathfrak{g}_n$ and $\mathfrak{l}_n=\mathfrak{p}_n/\mathfrak{n}_n$. From now on the composition of $\chi:\mathbb{C}^\times \to P$ and $P \twoheadrightarrow P/U=L$ will also be denoted by $\chi:\mb{C}^\times \to L$.

\subsection{Induction and restriction} 
Let's recall the induction diagram from \ref{cuspidal}
  \[\mathscr{N}_L\xleftarrow{\pi_P}\mathscr{N}_L+\mathfrak{u}_P\xrightarrow{e_P}G\times^{P}(\mathscr{N}_L+\mathfrak{u}_P)\xrightarrow{\mu_P}\mathscr{N}_G,
  \] where $\mathfrak{u}_P=\lie(U_P)$, $\pi_P, e_P$ are the obvious maps and $\mu_P(g,x)=Ad(g)x$. A slight modification of this diagram  gives us the induction diagram in the graded setting.

\[
\begin{tikzcd}
\mathfrak{l}_n & \mathfrak{p}_n \arrow[l, "\pi"'] \arrow[r, "e"] \arrow[rr, "i"', bend right] & G_0\times^{P_0}\mathfrak{p}_n \arrow[r, "\mu"] & \mathfrak{g}_n
\end{tikzcd}
\]
As before,  $\pi$ is  projection, $e,i$ are inclusions and $\mu(g,x)=Ad(g)x$. The induction functor is denoted by
\[\ind:D^b_{L_0}(\mathfrak{l}_n)\to D^b_{G_0}(\mathfrak{g}_n).\] As $P_0=L_0\ltimes U_0$ and $U_0$ acts on $\mathfrak{l}_n$ trivially,  we have equivalence of categories $D^b_{P_0}(\mathfrak{l}_n)\cong D^b_{L_0}(\mathfrak{l}_n)$. So instead of starting from $D^b_{L_0}(\mathfrak{l}_n)$ we can start from $D^b_{P_0}(\mathfrak{l}_n)$. So we define \[\ind(\mathcal{F}):={\mu}_!(\underbrace{e^*\f^{G_0}_{P_0}}_\text{Induction  Equivalence})^{-1}{\pi}^*(\mathcal{F}).\] 
\\
Here $e^*\f^{G_0}_{P_0}: D^b_{G_0}(G_0\times^{P_0}\mf{p}_n)\to D^b_{P_0}(\mf{p}_n)$ is the induction equivalence map, hence its inverse makes sense.
The definition of restriction also comes from the diagram above,
$\res:D^b_{G_0}(\mathfrak{g}_n)\to D^b_{L_0}(\mathfrak{l}_n)$ is defined by, $$\res(\mathcal{F}):={\pi}_!i^*\f^{G_0}_{L_0}(\mathcal{F}).$$
\begin{theorem}\label{ver}
The functor $\ind$ commutes with $\mb{D}$, the Verdier duality functor.
\end{theorem}
\begin{proof}
	The map $\mu$   is  proper, therefore it  commutes with $\mb{D}$. By \cite[prop. ~7.6.2]{BL}, $\mb{D}(e^*\f^{G_0}_{P_0})=(e^*\f^{G_0}_{P_0})\mb{D}[-2\dim G_0/{P_0}]$. The map $\pi$ is smooth and has relative dimension of $\dim \mf{p}_n-\dim \mf{l}_n=2\dim G_0/{P_0} $. Therefore $\mb{D}\pi^*=\pi^!\mb{D}=\pi^*\mb{D}[2\dim G_0/{P_0}]$. Combining all these facts we can see, $\mb{D}\ind=\ind \mb{D}$.
\end{proof}

 \subsection{Transitivity}
 Before going into the main result of this section we will talk about the transitivity of induction. Let $P$ be a parabolic subgroup of $G$ containing the Levi subgroup $L$ which contains $\chi(\mb{C}^\times)$. Let $R$ be a parabolic contained in $P$  with Levi $M\subset L$, which again contains $\chi(\mb{C}^\times)$. Then $R\cap L$ is  a parabolic subgroup of $L$ with the Levi factor $M$. Let $\mf{r,m}$ be the Lie algebras of $R,M$ respectively.
 
 \begin{theorem}
 Let $R\subset P$ and $M\subset L$ as defined above. Then for $\mc{F}\in D^b_{M_0}(\mf{m}_n)$, $\inn^{\mf{g}}_{\mf{p}}\inn^\mf{l}_{\mf{l}\cap\mf{r}}(\mc{F})\cong \inn^\mf{g}_\mf{r}(\mc{F})$.	
 \end{theorem}
 \begin{proof}
 The proof is clear from the diagram below.
\[
\begin{tikzcd}
\mathfrak{r}_n \arrow[d] \arrow[r] \arrow[dd, "\pi_R"', bend right=49] \arrow[rr, "i_R", bend left] & \mathfrak{p}_n \arrow[r] \arrow[d, "\pi_P"] & \mathfrak{g}_n \\
\mathfrak{l}_n\cap \mathfrak{r}_n \arrow[d, "\pi_{L\cap R}"] \arrow[r, "i_{L\cap R}"]               & \mathfrak{l}_n                              &                \\
\mathfrak{m}_n                                                                                      &                                             &               
\end{tikzcd}
\]
 \end{proof}

\subsection{Lusztig's original definition}\label{2.2}Lusztig's original definition of the restriction is same as we defined above. But for induction, he  used a different diagram.
\[
\mathfrak{l}_n\xleftarrow{p_1}E'\xrightarrow{p_2}E''\xrightarrow{p_3}\mathfrak{g}_n
\] 

where $E'=G_0\times^{U_0}\mathfrak{p}_n$ and $E''=G_0\times^{P_0}\mathfrak{p}_n$. 
Here $p_1(g,x)=\pi(x)$, $p_2$ is the obvious map and $p_3(g,x)=Ad(g)x$. Induction is defined by $\ind(\mathcal{F})={p_3}_!(\underbrace{{p_2^*\f^{G_0}_{P_0}}}_\text{Induction Equivalence})^{-1}p_1^*(\mathcal{F})$.

\begin{lemma}\label{match}
Lusztig's original definition of induction matches with the definition given here.	
\end{lemma}
\begin{proof}

 It follows from the diagram below.
\[
\begin{tikzcd}
\mathfrak{l}_n & \mathfrak{p}_n\arrow[l, "\pi"'] \arrow[r, "h"] & G_0\times{\mathfrak{p}_n} \arrow[r, "q_P"] \arrow[d, "q_U"]          & G_0\times^{P_0}\mathfrak{p}_n \arrow[r, "p_3=\mu"] & \mathfrak{g}_n \\
    &                                    & G_0\times^{U_0}\mathfrak{p}_n \arrow[ru, "p_2"'] \arrow[llu, "p_1"'] &                                     &    
\end{tikzcd} 
\]
Clearly,
\begin{align*}
{p_3}_!({{{p_2}^*\f^{G_0}_{P_0}}})^{-1}{p_1}^*(\mathcal{F})&={p_3}_!\underbrace{{\f^{G_0}_{P_0}}^{-1}{{q_P}^*}^{-1}{{q_U}^*}}_{({p_2}^*\f^{G_0}_{P_0})^{-1}}\underbrace{{{q_U}^*}^{-1}{h^*}^{-1}\pi^*}_{{p_1}^*}\\
&={p_3}_!((q_P\circ h)^*\f^{G_0}_{P_0})^{-1}\pi^*(\mathcal{F})\\
&={\mu}_!(e^*\f^{G_0}_{P_0})^{-1}\pi^*(\mathcal{F})\\
&=\ind(\mathcal{F}).
\end{align*}
\end{proof}
	So from now we can use any of the induction diagrams defined above.

\subsection{Cleanness for cuspidal pairs}

\begin{theorem}\label{th55}
$(\mathcal{O},\mathcal{L}) \in \mathscr{I}(\mathfrak{g}_n)^{\cu}$ is clean.
\end{theorem}                       

\begin{proof}
Let $(\mc{O},\mc{L})\in \ms{I}(\mf{g}_n)^{\cu}$ and  $(C,\mc{E})\in \ms{I}(G)^{0-\cu}$ so that $C\cap \mf{g}_n=\mc{O}$ and $\mc{L}=\mc{E}|_{\mc{O}}$.
	Note that $\mathscr{E}^\lor$ is also cuspidal by Remark \ref{rk1.2.1}.  Let $X$ be another $G_0$-orbit in $\mf{g}_n$ other than $\mathcal{O}$. We will show that $\mc{IC}(\mathcal{O},\mathcal{L})|_X=0$ and 
 $\mc{IC}(\mathcal{O},\mathcal{L}^\lor)|_X=0$. For descending induction, assume it is true for orbits $X'$, where $\dim(X)< \dim(X')< \dim(\mf{g}_n)$.
 Let $x\in X$. By Theorem \ref{th1}, we can find \[\phi :\mf{sl}_2 \to \mf{g} \text{ such that }\phi(e)=x\in \mf{g}_n, \phi(f)=x'\in \mf{g}_{-n}, \phi(h) \in \mf{g}_0\]where $e,f,h$ are defined in  the background. Let $\tilde{\phi}: SL_2\to G$ be such that $d\tilde{\phi}=\phi$. Define $\chi': \mb{C}^\times \to G$ by

\[\chi'(a) = \tilde{\phi} \begin{pmatrix} a & 0\\ 0 & a^{-1} \end{pmatrix}.\]

Let ${\mf{g}}^{x'}$ be the centralizer of $x'$ in $\mf{g}$. Let $\Sigma=x+{\mf{g}}^{x'}$ and  $\tilde{\Sigma}=\Sigma \cap \mf{g}_n$.
According to Slodowy \cite[pp. ~109]{Sw},
\begin{align}\label{7}
\Sigma \text{ is transversal to the }G \text{-orbit of }x\text{ in }\mf{g}.
\end{align}
Now $\mb{C}^\times$ acts on $G$ by conjugation via $\chi$ and on $\mf{g}$ by $a^{-n}Ad(\chi(a))$, call it $\psi$, which fixes $x$ and preserves $\Sigma$ as $x'\in \mf{g}_n$. The action of $G$ on $\mf{g}$ is $\mb{C}^\times$-equivariant. So we can restrict the action to the fixed point sets of $\mb{C}^\times$-actions 
  and  see that $G_0$ acts on  $\mf{g}_n$.
Using (\ref{7}) we deduce that 
\begin{align}\label{8}
\tilde{\Sigma}\text{ is transverse to the }G_0\text{-orbit of }x\text{ in }\mf{g}_n.
\end{align}

Now we define another action $\psi'$, $\mb{C}^\times$ acts on $\Sigma$ by  $a\to a^{-2}Ad(\chi'(a))$. This action is well-defined; if $x+y \in \Sigma$,  then $[y,x']=0$; so $[Ad(\chi'(a)y,Ad(\chi'(a))x']=0$. Let $c_{\chi'(a)}$ denote the conjugation by $\chi'(a)$. Now $Ad(\chi'(a))x'=d(c_{\chi'(a)})d\tilde{\phi}|_f=d(c_{\chi'(a)} \tilde{\phi})|_f=a^{-2}d\tilde{\phi}|_f=a^{-2}x'$. So we have $[Ad(\chi'(a))y, x']=0$. 
Also $Ad(\chi'(a))x=d(c_{\chi'(a)})\circ d\tilde{\phi}|_e=d(c_{\chi'(a)}\circ \tilde{\phi})|_e=a^2d\tilde{\phi}|_e=a^2x$. Hence $a^{-2}Ad(\chi'(a))(x+y) \in \Sigma$. Now we will show that  
\begin{align}\label{9}
\text{this action $\psi'$ stabilizes }\tilde{\Sigma}\text{  and }\mathcal{O}\cap \Sigma. 
\end{align}
To show the first part it is enough to show that if $y\in \mf{g}_n$, then $Ad(\chi'(a))y \in \mf{g}_n$, because we already have shown that $\Sigma$ is stable under this action. Now $\phi(h)\in \mf{g}_0$, so the Lie subalgebra generated by $\phi(h)$ is in $\mf{g}_0$. Thus by \cite[ Theorem ~13.1]{hum}, $\tilde{\phi}   \begin{pmatrix} a & 0\\ 0 & a^{-1} \end{pmatrix} \subset G_0$. Therefore $\chi'$ commutes with $\chi$                                        
 and we are done with the proof  that $Ad(\chi'(a))y \in \mf{g}_n$. If $y\in \mc{O}$, then $Ad(\chi'(a)y$ is  also in $\mc{O}$ as $Im( \chi')\subset G_0$ and $\mc{O}$ is a $G_0$-orbit. 
 
 Now we can consider a $\mf{sl}_2$ action on $\mf{g}$ by $(s,v)\in \mf{sl}_2\times \mf{g}$ goes to $[\phi(s), v]$. Via this action the Lie algebra generated by $\phi(h)$ acts on $\mf{g}^{x'}$. The unique lift of this action after multiplying by $t^{-2}$, where $t\in \mb{C}^\times$, gives rise to $\psi'$, that we talked already. Now in the original action $f$ acts on $\mf{g}^{x'}$ gives $0$. Therefore all the eigen values of the action of $h$ on $\mf{g}^{x'}$ will be negative.   
\begin{align}\label{10}
\text{Hence the action $\psi'$ is a repelling action   on }\Sigma\text{ to }x.
\end{align}
\begin{align}\label{11}
\text{ by Conjecture \ref{con}, } \mc{IC}({C},\mc{E})|_{\bar{{C}}-{C}}=0. 
\end{align}
As $\mc{E}^\lor$ is also cuspidal,  the same result is true for $\mc{E}^\lor$.

Using the transversal property of (\ref{7}) and the definition of transversal slice, the map $\mu:G\times \Sigma  \to \mf{g}$ is smooth of relative dimension $\dim G - \dim G\cdot x$.

Hence by \cite[pp. ~110]{BBD}, pullback with a shift takes $\mc{IC}$'s to $\mc{IC}$'s.
\[
\begin{tikzcd}
G\times \Sigma \arrow[r] \arrow[r] \arrow[r, "\mu"] & \mf{g}\\
 \Sigma \arrow[u, "h", hook] \arrow[ru, hook]        &        
\end{tikzcd}\] where $\mu $ is smooth and $h$ induces the induction equivalence, $D^b(\Sigma)\cong D^b_G(G\times \Sigma)$.
Hence from the above diagram, we can say that 
 \begin{align}\label{12}
 \mc{IC}({C},\mc{E})|_{\Sigma}=\mc{IC}({C}\cap \Sigma, \mc{E}|_{C\cap\Sigma})[m]
 \end{align}
  where $m=\dim{C}-\dim({C}\cap \Sigma)$. Similarly, $\mc{IC}({C},\mc{E}^\lor)|_{\Sigma}=\mc{IC}({C}\cap \Sigma, \mc{E}^\lor|_{C\cap \Sigma})[m]$.
  By (\ref{11}) and (\ref{12}), 
 \begin{align}\label{13}
 \mc{IC}({C}\cap \Sigma, \mc{E}|_{C\cap \Sigma})|_{(\bar{{C}}-{C})\cap \Sigma}=0\text{ and  }\mc{IC}({C}\cap \Sigma, \mc{E}^\lor|_{C\cap \Sigma})|_{(\bar{{C}}-{C})\cap \Sigma}=0. 
 \end{align}
\item Using the repelling action from (\ref{10}) and  Lemma \ref{lm3.9}, \cite[Theorem ~1]{BR}, we get, 
\begin{align}\label{14}
\begin{cases} \mc{IC}({C}\cap \Sigma, \mc{E}|_{C\cap \Sigma})_x=R\Gamma(\mc{IC}({C}\cap \Sigma, \mc{E}|_{C\cap \Sigma}),\text{ and}\\
\mc{IC}({C}\cap \Sigma, \mc{E}^\lor|_{C\cap \Sigma})_x=R\Gamma(\mc{IC}({C}\cap \Sigma, \mc{E}^\lor|_{C\cap \Sigma}).
\end{cases}
\end{align}
 Now $x\in (\bar{{C}}-{C})\cap \Sigma$, so from (\ref{13}), 
\begin{align}\label{15}\mc{IC}({C}\cap \Sigma, \mc{E}|_{C\cap\Sigma})_x=0\text{. So by (\ref{14}), } R\Gamma(\mc{IC}({C}\cap \Sigma, \mc{E}|_{C\cap \Sigma}))=0.
\end{align}This implies,
\begin{align}\label{16}
\begin{cases}
R\Gamma_c(\mc{IC}({C}\cap \Sigma,\mc{E}^\lor|_{C\cap \Sigma}))&= R\Gamma_c(\mathbb{D}\mc{IC}({C}\cap \Sigma,\mc{E}|_{C\cap \Sigma}))\\ &=\mb{D}R\Gamma(\mc{IC}({C}\cap \Sigma, \mc{E}|_{C\cap \Sigma}))=0.
\end{cases}
\end{align}
Similarly, 
\begin{align}\label{17}
R\Gamma(\mc{IC}({C}\cap \Sigma, \mc{E}^\lor|_{C\cap \Sigma})=0\text{ and }R\Gamma_c(\mc{IC}({C}\cap \Sigma,\mc{E}|_{C\cap \Sigma}))=0.
\end{align}
Now we claim  that 
\begin{align}\label{18}
R\Gamma_c(\mc{E}^\lor|_{C\cap \Sigma})=0.
\end{align} From the open-closed embedding,
\[
{C}\cap \Sigma\xhookrightarrow{j} \bar{C}\cap\Sigma\xhookleftarrow{i}(\bar{{C}}-{C})\cap \Sigma\]  gives us the distinguished triangle,
 
 \[j_!j^*\mc{IC}({C}\cap \Sigma,\mc{E}^\lor|_{C\cap \Sigma})\to \mc{IC}({C}\cap \Sigma,\mc{E}^\lor|_{C\cap \Sigma}) \to i_*i^*\mc{IC}({C}\cap \Sigma,\mc{E}^\lor|_{C\cap \Sigma})\to\hspace{5mm}.
 \] We can apply $R\Gamma_c$ to get
 \[
 R\Gamma_c(j_!j^*\mc{IC}({C}\cap \Sigma,\mc{E}^\lor|_{C\cap \Sigma}))\to R\Gamma_c(\mc{IC}({C}\cap \Sigma,\mc{E}^\lor|_{C\cap \Sigma})) \to R\Gamma_c(i_*i^*\mc{IC}({C}\cap \Sigma,\mc{E}^\lor|_{C \cap \Sigma}))\to\hspace{5mm}.
 \] The first term in this distinguished triangle is $R\Gamma_c(\mc{E}^\lor|_{C\cap \Sigma})$ with a shift. The second term is $0$ by (\ref{16}) and third term is $0$ by (\ref{13}), hence  (\ref{18}) is proved.
\item     
 For the action $\psi$ of $\mb{C}^\times$  on ${C}\cap \Sigma$ by $a \to a^{-n}Ad(\chi(a))$, the fixed point set is $\mathcal{O}\cap \tilde{\Sigma}$. 
 So by Lemma \ref{lm9},  
 \begin{align}\label{19}
 R\Gamma_c({\mc{L}}^\lor|_{\mc{O}\cap \tilde{\Sigma}}  )=0.
 \end{align}    
  By the transversal property (\ref{8}), we have, 
  \begin{align}\label{20}
  \mc{IC}(\mathcal{O}, \mathcal{L})|_{\tilde{\Sigma}}=\mc{IC}(\mathcal{O}\cap \tilde{\Sigma}, {\mc{L}}_{\mc{O}\cap \tilde{\Sigma}})[n]
  ,\end{align} and
  \begin{align}\label{20'}
  \mc{IC}(\mathcal{O}, \mathcal{L}^\lor)|_{\tilde{\Sigma}}=\mc{IC}(\mathcal{O}\cap \tilde{\Sigma}, {\mc{L}^\lor}_{\mc{O}\cap \tilde{\Sigma}})[n],
  \end{align} where  $n=\dim \mathcal{O}- \dim \mathcal{O}\cap \tilde{\Sigma}$.
  
  By repelling property (\ref{10}) we have, 
  \begin{align}\label{21}
  \mc{IC}(\mathcal{O}\cap \tilde{\Sigma}, {\mc{L}}|_{\mc{O}\cap \tilde{\Sigma}})_x= R\Gamma(\mc{IC}(\mathcal{O}\cap \tilde{\Sigma},    {\mc{L}}|_{\mc{O}\cap \tilde{\Sigma}}))
  \end{align}and
  \begin{align}\label{22}
  \mc{IC}(\mathcal{O}\cap \tilde{\Sigma}, {\mc{L}}^\lor|_{\mc{O}\cap \tilde{\Sigma}})_x= R\Gamma(\mc{IC}(\mathcal{O}\cap \tilde{\Sigma},{\mc{L}}^\lor|_{\mc{O}\cap \tilde{\Sigma}})).
  \end{align}
  
  Here $\overline{\mathcal{O}\cap\tilde{\Sigma}}-\{x\}$ is the union of $V\cap \tilde{\Sigma}$, where each $V$ is a $G_0$ orbit whose closure contains $x$, hence also $X$. So $\dim{V}> \dim{X}$. Also $\overline{\mathcal{O}\cap\tilde{\Sigma}}\cap X=\{x\}$. So we can use the induction hypothesis on $\overline{\mathcal{O}\cap\tilde{\Sigma}}-(\mathcal{O}\cap \tilde{\Sigma})-\{x\}$ and (\ref{20'}),  therefore, 
  \begin{align}\label{23}
  \mc{IC}(\mathcal{O}\cap\tilde{\Sigma},{\mc{L}}^\lor|_{\mc{O}\cap\tilde{\Sigma}})\text{ is }0\text{ on }\overline{\mathcal{O}\cap\tilde{\Sigma}}-(\mathcal{O}\cap \tilde{\Sigma})-\{x\}.
 \end{align} 
Now we use the open and closed embeddings below for $\mc{IC}(\mathcal{O}\cap \tilde{\Sigma},{\mc{L}}^\lor|_{\mc{O}\cap \tilde{\Sigma}})|_{ \overline{\mathcal{O}\cap \tilde{\Sigma}}-\{x\} }$, 
  \[ \mathcal{O}\cap\tilde{ \Sigma}\xhookrightarrow{j} \overline{\mathcal{O}\cap \tilde{\Sigma}}-\{x\}\xhookleftarrow{i} \overline{\mathcal{O}\cap\tilde{\Sigma}}-(\mathcal{O}\cap \tilde{\Sigma})-\{x\}.
  \] This gives us the distinguished triangle,
  \[j_!j^* \mc{IC}(\mathcal{O}\cap \tilde{\Sigma},{\mc{L}}^\lor|_{\mc{O}\cap \tilde{\Sigma}})|_{ \overline{\mathcal{O}\cap \tilde{\Sigma}}-\{x\}}\to \mc{IC}(\mathcal{O}\cap \tilde{\Sigma},{\mc{L}}^\lor|_{\mc{O}\cap \tilde{\Sigma}})|_{ \overline{\mathcal{O}\cap \tilde{\Sigma}}-\{x\}} \to i_*i^* \mc{IC}(\mathcal{O}\cap \tilde{\Sigma},{\mc{L}}^\lor|_{\mc{O}\cap \tilde{\Sigma}})|_{ \overline{\mathcal{O}\cap \tilde{\Sigma}}-\{x\}}\to.\] 
We have, 
\begin{equation}\label{eq20''}
R\Gamma_c(\mc{IC}(\mathcal{O}\cap \tilde{\Sigma},{\mc{L}}^\lor|_{\mc{O}\cap \tilde{\Sigma}}))|_{\overline{\mathcal{O}\cap \tilde{\Sigma}}-\{x\}}=0
\end{equation}
as the first term in the distinguished triangle, $R\Gamma_c(\mc{L}^\lor|_{\mc{O}\cap \tilde \Sigma})$ vanishes by (\ref{19}) and the third term vanishes by (\ref{23}).
  
  \item Now from the open-closed embedding,
  
  \[\{x\}\xhookrightarrow{i}\overline{\mathcal{O}\cap \tilde{\Sigma}}\xhookleftarrow{j} \overline{\mathcal{O}\cap \tilde{\Sigma}}-\{x\}\]
  we get, 
  \[j_!j^*\mc{IC}(\mathcal{O}\cap \tilde{\Sigma},{\mc{L}}^\lor|_{\mc{O}\cap \tilde{\Sigma}})\to \mc{IC}(\mathcal{O}\cap \tilde{\Sigma},{\mc{L}}^\lor|_{\mc{O}\cap \tilde{\Sigma}}) \to i_*i^* \mc{IC}(\mathcal{O}\cap \tilde{\Sigma},{\mc{L}}^\lor|_{\mc{O}\cap \tilde{\Sigma}})
  .\]By (\ref{eq20''}),
\begin{align}\label{24}
\mc{IC}(\mathcal{O}\cap \tilde{\Sigma}, \mc{L}^\lor|_{\mc{O}\cap \tilde{\Sigma}})_x=R\Gamma_c \mc{IC}(\mathcal{O}\cap \tilde{\Sigma}, \mc{L}^\lor|_{\mc{O}\cap \tilde \Sigma}).
\end{align}
    
  \item  From (\ref{21}), \[
   \mathbb{D}\mc{IC}(\mathcal{O}\cap \tilde{\Sigma}, {\mc{L}}|_{\mc{O}\cap \tilde{\Sigma})_x}= \mathbb{D}R\Gamma(\mc{IC}(\mathcal{O}\cap \tilde{\Sigma},{\mc{L}}|_{\mc{O}\cap \tilde {\Sigma}}))
     =
                                                               R\Gamma_c\mc{IC}(\mathcal{O}\cap \tilde{\Sigma}, {\mc{L}}^\lor|_{\mc{O}\cap \tilde{\Sigma}})
  \]
Hence from (\ref{24}), we get,  \[
\mathbb{D}\mc{IC}(\mathcal{O}\cap \tilde{\Sigma}, {\mc{L}}|_{\mc{O}\cap \tilde{\Sigma})_x}= \mc{IC}(\mathcal{O}\cap \tilde{\Sigma}, {\mc{L}}^\lor|_{\mc{O}\cap \tilde{\Sigma}})_x.
\]  Since $\mc{IC}(\mathcal{O}\cap \tilde{\Sigma}, {\mc{L}}^\lor|_{\mc{O}\cap \tilde{\Sigma}})_x$ lives in degrees $<0$. Hence $\mathbb{D}\mc{IC}(\mathcal{O}\cap \tilde{\Sigma}, {\mc{L}}^\lor|_{\mc{O}\cap \tilde{\Sigma}})_x$ lives in degrees $>0$. But $\mc{IC}(\mathcal{O}\cap \tilde{\Sigma}, {\mc{L}}^\lor|_{\mc{O}\cap \tilde{\Sigma}})_x$ again lives in degrees $<0$, which is a contradiction. So $\mc{IC}(\mathcal{O}\cap \tilde{\Sigma}, {\mc{L}}^\lor|_{\mc{O}\cap \tilde{\Sigma}})_x=0$ and by (\ref{20}), $\mc{IC}(\mathcal{O},\mathcal{L})_x=0$. Hence we are done.
\end{proof}
\begin{corollary}\label{existance for cuspidal}
For $(\mathcal{O},\mathcal{L}) \in \mathscr{I}(\mathfrak{g}_n)^{\cu}$, the parity sheaf $\E(\mathcal{O},\mathcal{L})$ exists and $\IC(\mathcal{O},\mathcal{L})=\E(\mathcal{O},\mathcal{L})$.	
\end{corollary}
\begin{proof}
	From the previous theorem $(\mathcal{O},\mathcal{L})$ is clean, i.e  $\IC(\mathcal{O},\mathcal{L}) $  restricted to $\bar{\mathcal{O}}-\mathcal{O}$ is $0$. So $\IC(\mathcal{O},\mathcal{L})=j_!\mc{L}[\dim{\mathcal{O}}] $, where $j:\mathcal{O}\xhookrightarrow{} g_n$, which obviously satisfies the parity condition. Hence by uniqueness of an indecomposable parity complex, 
$\IC(\mathcal{O},\mathcal{L})=\E(\mathcal{O},\mathcal{L})$.
\end{proof}

  \section{Induction preserves parity for cuspidal pairs}\label{sec5}
 \subsection{Parabolic induction diagram for cuspidal pairs}\label{subsec5.2}
Recall the parabolic induction diagram we defined in \ref{cuspidal},

\begin{equation}\label{26}
\mathscr{N}_L\xleftarrow{\pi_P}\mathscr{N}_L+\mathfrak{u}_P\xrightarrow{e_P}G\times^{P}(\mathscr{N}_L+\mathfrak{u}_P)\xrightarrow{\mu_P}\mathscr{N}_G
\end{equation}
If instead of working with a general pair in $\ms{I}(L)$, we work with a pair $(C,\mc{E})\in \ms{I}(L)^{0-\cu}$, then we can do a slight modification to our standard parabolic induction diagram. 
We define a diagram associated with $P,L,C,\mc{E}$,
\[C\xleftarrow{a} C+\mf{u}_P\xrightarrow{b} G\times^P(C+\mf{u}_P)\xrightarrow{c}\mc{N}_G,\]Here $b$ is the obvious map and \[a(x)=\pi_P(x), c(g,x)=Ad(g)x.\] Now we define $\Ind^G_P$ for the cuspidal pair as 
\begin{equation}\label{27}
\Ind^G_P(\mc{IC}(C,\mc{E}))=c_!(b^*\f^G_P)^{-1}a^*\mc{E}[\dim C],
\end{equation} where $(b^*\f^G_P)^{-1}$ is the induction equivalence map, hence inverse makes sense. 
\begin{lemma}
This definition of parabolic induction for cuspidal pairs coincides with the definition we first gave.	
\end{lemma}
\begin{proof}
	
By Conjecture \ref{con}, $(C,\mc{E})$ is clean and  using the commutative diagram below we can show this definition of  parabolic induction coincides with the original one for cuspidal pairs. 
\[\begin{tikzcd}
\mathcal{N}_L     & \mathcal{N}_L+\mathfrak{u}_P \arrow[l, "\pi_P"']  \arrow[r, "e_P"]  & G\times ^P(\mathcal{N}_L+\mathfrak{u}_P) \arrow[r, "\mu_P"] & \mathcal{N}_G      \\
C \arrow[u, hook] &                                                                                                 C+\mf{u}_P \arrow[l, "a"'] \arrow[u, hook] \arrow[r, "b"]            & G\times^P(C+\mf{u}_P) \arrow[u, hook] \arrow[r, "c"]                          & \mc{N}_G \arrow[u, "="]
\end{tikzcd}
\]
\end{proof}
 \subsection{Induction diagram for cuspidal pairs}\label{3.1}
 In this section, we  first redefine Lusztig's induction diagram for cuspidal pairs. Let $P$ be a parabolic subgroup of $G$ and $L,U$ be its Levi subgroup and the unipotent radical, respectively. Let $(\mc{O},\mc{L})\in \i'$.  We define the induction diagram to be,
  \begin{equation}
 	\begin{tikzcd}
\mc{O} & \mc{O}+{\mf{u}}_n \arrow[l, "p'_1"'] \arrow[r, "p'_2"] & G_0\times^{P_0}(\mc{O}+{\mf{u}}_n )\arrow[r, "p'_3"] & \mf{g}_n .
\end{tikzcd}
 \end{equation}        
We define $p'_3:G_0\times^{P_0}(\mc{O}+{\mf{u}}_n)$ $\to \mf{g}_n$ to be $p'_3(g,z)=Ad(g)z$, 
  \[p'_2:\mc{O}+{\mf{u}_n} \to G_0\times^{P_0}(\mc{O}+{\mf{u}}_n)\text {, }p'_2\text{ to be the obvious map and, }
\]
  $p'_1: \mc{O}+{\mf{u}}_n$ $\to \mc{O}$ to be $p'_1(z)=\pi(z)$, where $\pi:\mf{p}_n\to \mf{l}_n$.  
  We start with $(\mathcal{O},\mathcal{L})\in \ms{I}(\mf{l}_n)^{\cu}$ and redefine the induction diagram.
We define \[\ind(\mc{IC}(\mc{O},\mc{L}))={p'_3}_!(p_2'^*\f^{G_0}_{L_0})^{-1}p_1'^*(\mathcal{L}[\dim \mc{O}]).\] 
\begin{lemma}
This definition of induction for cuspidal pairs coincides with Lusztig's original definition.	
\end{lemma}
\begin{proof}
	 By Theorem \ref{th55}, $(\mc{O},\mc{L})$ is clean and it coincides with the Lusztig's definition of induction because of the following commutative diagram.\[\begin{tikzcd}
\mathcal{O} \arrow[d, hook] & \mc{O}+\mathfrak{u}_n \arrow[l, "p_1'"'] \arrow[r, "p_2'"] \arrow[d, hook] & G_0\times^{P_0}(\mc{O}+{\mathfrak{u}}_n )\arrow[r, "p_3'"] \arrow[d, hook] & \mathfrak{g}_n \arrow[d] \\
\mathfrak{l}_n              & E' \arrow[l, "p_1"'] \arrow[r, "p_2"]                                     & E'' \arrow[r, "p_3"]                                    & \mathfrak{g}_n          
\end{tikzcd}\]

\end{proof}

\subsection{Parity preserved for cuspidal pairs}

 \begin{theorem}\label{th5}
 Let $P$ be a parabolic subgroup of G and $L$ be its Levi subgroup. If $(\mathcal{O}, \mathcal{L})\in \i'$, then $\ind(\mathcal{E}(\mathcal{O}, \mathcal{L}))$ is parity.
 	
 \end{theorem}
\begin{proof}
By corollary \ref{existance for cuspidal}, $\mc{E}(\mc{O},\mc{L})$ exists.
Let $(C,\mc{E})\in \ms{I}(L)^{0-\cu}$ be such that, $C\cap \mf{g}_n=\mc{O}$ and $\mc{E}|_{\mc{O}}=\mc{L}$. 
Let $y\in \mf{g}_n$ and 
$c:G\times^P(C+\mf{u}_P)\to \mc{N}_G$ be the map introduced in the previous subsection. Let $Y_y=c^{-1}(y)$. Then we have an isomorphism $ G/P\times (C+\mf{u}_p)  \to G\times^P (C+\mf{u}_P) $ given by,
\[(gP,x)\mapsto (g,Ad(g^{-1})x).\] It is easy to check that under this isomorphism the map $c$ becomes,
\[(gP,x)\mapsto x,\] and $Y_y=\{gP\in G/P| Ad(g^{-1})y\in \pi_P^{-1}(C)\}$. 
Recall that nilpotent $G$-orbits in $\mf{g}$ are all even dimensional.
Then using the definition of induction from (\ref{27}) and from the base change diagram below, and Conjecture \ref{2.6(c)}, we get, \[\h^a_c(Y_y,(b^*\f^G_P)^{-1}a^*\mc{E}[\dim C]|_{Y_y})=0\text{ for $a$ odd.}\]

$$\begin{tikzcd}
G\times^P(C+\mf{u}_P) \arrow[r, "c"]                   & \mc{N}_G\\
Y_y \arrow[r, "c"] \arrow[u, hook] & y \arrow[u]
\end{tikzcd}$$
We define an action of $\mb{C}^\times$ on $Y_y$ by, $(t,gP)\to \chi(t)gP$. This is well defined as $y\in \mf{g}_n$. Let $(Y_y)^{\mb{C}^\times}$ be the fixed point set.
From Theorem \ref{th7}, 
\begin{equation}\label{29}
\h^a_c((Y_y)^{\mb{C}^\times},(b^*\f^G_P)^{-1}a^*\mc{E}[\dim C] |_{(Y_y)^{\mb{C}^\times}})=0\text{,  for $a$ odd.}
\end{equation}


We will show that $(Y_y)^{\mb{C}^\times}=\sqcup_i{Z^i}$, where $P^i, i\in[1,b] $, is defined to be a set of  representatives of $G_0$-orbits of parabolic subgroups in $G$ conjugate to $P$ containing $\chi(\mb{C}^\times)$. Let $L^i$ and $U_{P^i}$ be the Levi and the unipotent radical of $P^i$ respectively. An element of $G$ conjugates $P$ to $P^i$, conjugating $C$ by the same element gives $C^i$ contained in $\mf{l}^i$. Let 
\[Z^i=\{g(P^i)_0\in G_0/(P^i)_0| Ad(g^{-1})y\in (\pi^i)^{-1}\mathcal(C^i)\},\]
 where $\pi^i: \mf{p}^i\to \mf{l}^i$ and $(P^i)_0=P^i\cap G_0$. We want to identify $g(P^i)_0\in Z^i$ with $gg'P$ in $(Y_y)^{\mb{C}^\times}$, where $g'\in G$ is fixed and $g'Pg'^{-1}=P^i$.  Note $g(P^i)_0\in Z^i$, so $Ad(g^{-1})y\in (\pi^i)^{-1}({C}^i)$, hence 
 \[
 Ad(gg')^{-1}y=Ad(g')^{-1}Ad(g^{-1})y\in Ad(g'^{-1}) (\pi^i)^{-1}({C}^i),\] 
which is by definition $\pi_P^{-1}({C})$. Also, \[
 (gg')^{-1}\chi(t)gg'=g'^{-1}g^{-1}\chi(t)gg'=g'^{-1}\chi(t)g',\] and $\chi(t)\in P^i$. Therefore by definition of $g'$, $g'^{-1}\chi(t)g'$ belongs to $ P$. By definition, 
 \[(Y_y)^{\mb{C}^\times}=\{gP\in G/P| Ad(g^{-1})y\in \pi_P^{-1}{C},g^{-1}\chi(t)g\in P\}.\] 
 Hence $gg'P$ is in $(Y_y)^{\mb{C}^\times}$.
Conversely, if $hP\in (Y_y)^{\mb{C}^\times}$, then $h^{-1}\chi(t)h\in P$ We can define $P^i=hPh^{-1}$ and $g'=h, g=e$, then by definition $gg'=h$ and $y=Ad(h)Ad(h^{-1})y\in Ad(h)\pi^{-1}{C}=(\pi^i)^{-1}{C}^i$, hence $e(P^i)_0\in Z^i$ by identifying this with $hP\in (Y_y)^{\mb{C}^\times}$. In the definition of $Z^i$, the condition $Ad(g^{-1})y\in (\pi^i)^{-1}({C}^i)$ can be redefined as below.

  If $y\in \mf{g}_n$ and $g\in G_0$, then this implies $Ad(g^{-1})y \in \mf{g}_n$. Hence we can restate the condition $Ad(g^{-1})y\in (\pi^i)^{-1}(C^i)$
 as, 
 $$Z^i=\{g(P^i)_0\in G_0/(P^i)_0 | Ad(g^{-1})y\in \mf{p}^i_n, \pi^i(Ad(g^{-1})y)\in \mc{O}^i\},$$ where $\mc{O}^i=C^i\cap \mf{g}_n$. 
In the redefined induction diagram above, if we use the isomorphism
\[G_0/P_0 \times (\mc{O}+\mf{u}_n) \xrightarrow {\cong} G_0\times^{P_0} (\mc{O}+\mf{u}_n)\]

we can see $Z^i={(p'_3)^{i}}^{-1}(y)$, where  $(p'_3)^{i}$ is the map associated to $(P^i,L^i)$ similar to how we defined $p'_3$.  Hence from base change and the diagram below,
\begin{equation}
\begin{cases}
	\h^a_c(Z^i,(b^*\f^G_P)^{-1}a^*\mc{E}[\dim C] |_{Z^i}  )&=\h^a({p'_3}_!({p'_2}^*\f^{G_0}_{P_0})^{-1}{p'_1}^*(i_!\mc{E}[\dim C])|_\mc{O})_y\\
	&=\h^a(\inn^\mf{g}_{\mf{p}^i}(\mc{IC}(\mathcal{O},\mathcal{L})[\dim C-\dim \mc{O}])_y
	\end{cases}
\end{equation} 

\begin{tikzcd}
                             &                                            &                                                                           & Z^i={(p'_3)^i}^{-1}\{y\} \arrow[d, hook] \arrow[r]      & \{y\} \arrow[d, hook]               \\
                             & \mathcal{O}^i\arrow[d, hook] & \mc{O}^i+\mf{u}^i_n \arrow[l, "{(p'_1)^i}"'] \arrow[d, ] \arrow[r, "{(p'_2)^i}"]            & G_0\times^{(P^i)_0}(\mc{O}^i+\mf{u}^i_n) \arrow[d, ] \arrow[r, "{(p'_3)^i}"] & \mf{g}_n\cap\mc{N}_G \arrow[d, hook] \\
{C}^i \arrow[r, "i"] & \mf{l}                                          & C^i+\mf{u}_{P^i} \arrow[l, "a'"'] \arrow[r, "b^i"] \arrow[ll, "a^i" description, bend left] & G\times^{P^i}(C^i+\mf{u}_{P^i}) \arrow[r, "c^i"]                           & \mc{N}_G                        
\end{tikzcd}

we have, 
\begin{align*}
\h^a_c((Y_y)^{\mb{C}^\times},(b^*\f^G_P)^{-1}a^*\mc{E}[\dim C] |_{(Y_y)^{\mb{C}^\times}})&=\bigoplus_i {\h^a_c(Z^i,(b^*\f^G_P)^{-1}a^*\mc{E}[\dim C] |_{Z^i})}\\&=\bigoplus_i{\h^a(\inn^\mf{g}_{\mf{p}^i}(\mc{IC}(\mathcal{O},\mathcal{L})[\dim C-\dim \mc{O}])_y}.
\end{align*}
So we finally get,
\begin{align*}
	\h^a_c((Y_y)^{\mb{C}^\times},(b^*\f^G_P)^{-1}a^*\mc{E} |_{(Y_y)^{\mb{C}^\times}}[\dim C])= \bigoplus_i{\h^{a+\dim C-\dim \mc{O}}(\inn^\mf{g}_{\mf{p}^i}(\mc{IC}(\mathcal{O},\mathcal{L})))_y}.
\end{align*}

In
the last sum one of these $\mf{p}^i$ is our original $\mf{p}$. Hence from (\ref{29}), we can conclude that
\begin{equation}\label{*'}
 \h^{a+\dim C-\dim \mc{O}}(\ind(\mc{IC}(\mathcal{O},\mathcal{L})))_y =0\text{, for $a$ odd.}
\end{equation}
 If $\dim C-\dim \mc{O}$ is odd then $\inn^{\mf{g}}_{\mf{p}}\mc{IC}(\mc{O},\mc{L})$ is $*$-odd and if  $\dim C-\dim \mc{O}$ is even then $\inn^{\mf{g}}_{\mf{p}}\mc{IC}(\mc{O},\mc{L})$ is $*$-even.  As $\mc{IC}(\mc{O},\mc{L}^\lor)$ is also cuspidal, so 
\begin{equation}\label{**'}
\h^{a+\dim C-\dim \mc{O}}(\ind(\mc{IC}(\mathcal{O},\mathcal{L}^\lor)))_y =0\text{, for $a$ odd.}
\end{equation}
 \begin{align*}\text{But, }
\ind(\mc{IC}(\mathcal{O},\mathcal{L}^\lor))&= \ind(\mb{D}\mc{IC}(\mathcal{O},\mathcal{L}))\\ &= \mb{D}\ind(\mc{IC}(\mathcal{O},\mathcal{L}))\text{ (by Theorem \ref{ver}).}
\end{align*}
Therefore,

\begin{align*}
\h^{a+\dim C-\dim{ \mc{O}}}(j^!\ind(\mc{IC}(\mathcal{O},\mathcal{L})) )&= \h^{a+\dim C-\dim \mc{O}}(j^!\ind(\mb{D}\mc{IC}(\mathcal{O},\mathcal{L}^\lor)) )\\ &=  \h^{a+\dim C-\dim \mc{O}}(j^!\mb{D}\ind(\mc{IC}(\mathcal{O},\mathcal{L}^\lor)) )\\&= \h^{a+\dim C-\dim \mc{O}}(\ind(\mc{IC}(\mathcal{O},\mathcal{L}^\lor)))_y, 
\end{align*}

where $j:\{y\} \hookrightarrow \mf{g}_n\cap \mc{N}_G$. So by (\ref{**'}),  \[\h^{a+\dim C-\dim \mc{O}}(j^!\ind(\mc{IC}(\mathcal{O},\mathcal{L}^\lor)))=0\text{, for $a$ odd.}\]
  Hence by the above fact,  if $\dim C-\dim \mc{O}$ is odd then $\inn^{\mf{g}}_{\mf{p}}\mc{IC}(\mc{O},\mc{L})$ is $!$-odd and if  $\dim C-\dim \mc{O}$ is even then $\inn^{\mf{g}}_{\mf{p}}\mc{IC}(\mc{O},\mc{L})$ is $!$-even, finally we can say $\ind\mc{IC}(\mc{O},\mc{L})$ satisfies the parity condition.
\end{proof}

 \section{Existence of  parity sheaves}\label{sec6}
 
 \subsection{$n$-rigidity}\label{n-rigidity}
Let $n\in \mb{Z}$ be fixed. Recall the cocharacter map $\chi: \mb{C}^\times \to G$. Let $\phi:\mf{sl}_2\to \mf{g}$ and $\tilde\phi:SL_2\to G$ be such that $d\tilde\phi=\phi$. Define 
$\chi': \mb{C}^\times\to G$ by,
\[\chi'(t)=\tilde\phi\begin{pmatrix}
t&0\\0& t^{-1}\\	
\end{pmatrix}.
\]Now we define for $m\in \mb{Z}$, \[_m\mf{g}=\{x\in \mf{g}| Ad\chi'(t)x=t^m x\}. \]
Hence $\mf{g}={\bigoplus_{m\in \mb{Z}}} _m\mf{g}$.  
\begin{definition}
$(G,\chi)$ is said to be $n$-rigid if there exists $\phi$ such that
\begin{enumerate}
\item
$\phi\in J_n$, which is defined in section \ref{sec2}.
\item $_m\mf{g}=\mf{g}_{nm/2}$ for $m\in\mb{Z}$ and $nm/2\in \mb{Z}$,
\item $_m\mf{g}=0$ for $m\in \mb{Z}$ and $nm/2\notin \mb{Z}$.
	
\end{enumerate}
	
\end{definition}
\begin{proposition}\label{nrigid}
	If $(G,\chi)$ is $n$-rigid and $\phi(e)=x$, then
	\begin{enumerate}
		\item $x$ is in the unique open $G_0$-orbit in $\mf{g}_n$,
		\item the map $G_0^x/(G_0^x)^\circ\to G^x/(G^x)^\circ$ is an   isomorphism.
	\end{enumerate}
\end{proposition}
The proof of this proposition is given in \cite[prop ~4.2,5.8]{Lu}.
For the proof of Theorem \ref{th1.6} $n$-rigidity  plays a role. 

 \subsection{Construction of parabolic, nilpotent and Levi subgroups}\label{subsec6.1}
 In this section we  first fix $x \in \mf{g}_n$ and then we  construct $\mf{p,n,l}$ associated to $x$. From Theorem \ref{th55}, recall $\phi$ and the construction of $\chi'$. Recall $\chi$ commutes with $\chi'$ 
  and, 
  \[_m\mf{g}=\{g\in \mf{g}| Ad(\chi'(t))g=t^m g\}.\]
  Now we have the direct sum decomposition,
  \[
  \mf{g}={\bigoplus_{m,m'\in \mb{Z}}}_{m}\mf{g}_{m'}
  .\] 
 Here $m,m' \in \mathbb{Z}$ and $_{m}\mf{g}_{m'}=_{m}\mf{g} \cap \mf{g}_{m'}$. We define,
\[
\mf{p}=\bigoplus_{m',m, 2m'/n\leq m}(_m\mf{g}_{m'}),   \mf{n}=\bigoplus_{m',m, 2m'/n<m}(_m\mf{g}_{m'}), 
\mf{l}= \bigoplus_{m',m, 2m'/n=m}(_m\mf{g}_{m'}).
 \]
Here, $\mf{p,n,l}$ are parabolic, nilradical and Levi subalgebra of $\mf{g}$  \cite[~ 5]{Lu}.  We give one example  in \ref{subsec8.1}, how to construct $\mf{p,n,l}$ as defined here.
 \begin{theorem}\label{th5.1}
 	With  the set-up above, $\phi(\mf{sl}_2)\subset \mf{l}$ and $(L,\chi)$ is $n$-rigid. Also $x$ is in the open $L_0$-orbit in $\mf{l}_n$.
 \end{theorem}


 \subsection{Existence of parity sheaves}\label{subsec6.2}
  Let $(\mathcal{O},\mathcal{L})\in \mathscr{I}(\mf{g}_ n)$. Let $x\in \mathcal{O}$ and $\mf{p,n,l}$ be Lie subalgebras of $\mf{g}$ constructed as above connected with $x$. Let $P,U,L$ be the subgroups of $G$ with Lie algebras $\mf{p,n,l}$ respectively. It follows that P contains the image $\chi(\mb{C}^\times)$. 
  By Theorem \ref{th5.1}, $x$ is contained in an open $L_0$-orbit in $\mf{l}_n$, call it $\mc{O}_L$.
 As $x\in \mc{O}_L$,  $\mc{O}_L\subset \mc{O}$. 
   Now we can restrict $\mathcal{L}$ to $\mc{O}_L$. 
   By \cite[ prop ~5.8]{Lu}, the inclusion induces isomorphisms on $G_0^x/(G_0^x)^\circ$ and $L_0^x/({L_0^x})^\circ$. 
  Hence, \[\loc_{,G_0}(\mc{O},\Bbbk)=\loc_{,L_0}(\mc{O}_L,\Bbbk).\] So $\mathcal{L}|_{\mc{O}_L}$ is a local system on $\mf{l}_n$, let's call it $\mc{L}'$.  
  Now we are ready to prove the parity vanishing theorem  mentioned in section \ref{sec2}.
  \begin{theorem}\label{th6.2}
  Let $\mc{O}$ be a nilpotent orbit in $\mf{g}_n$ and $\mc{L}\in \loc_{,G_0}(\mc{O},\Bbbk)$, then $\h^*_{G_0}(\mc{L})$ vanishes in odd degrees.	
  \end{theorem}
\begin{proof}
	Define $\tilde{\mc{O}}:=G_0/(G^x_0)^\circ$ and $\pi:\tilde{\mc{O}}\to \mc{O}$ by $g(G^x_0)^\circ \to g.x$. This is a Galois covering map with the Galois group $A_{G_0}(x)$. Also we know that
	\[\loc_{, G_0}(\mc{O},\Bbbk)\cong \Bbbk[A_{G_0}(x)]-\text{mod}.\]We can construct the Levi subgroup $L$ as defined above. Then by Theorem \ref{th5.1}, $(L,\chi)$ is $n$-rigid. Hence by Proposition \ref{nrigid}, $L^x/(L^x)^\circ \cong L^x_0/(L^x_0)^\circ$. But in the above paragraph we mentioned $G^x_0/(G^x_0)^\circ \cong L^x_0/(L^x_0)^\circ$.
 By  Remark \ref{rk2.4.2}, $|A_L(x)|$ is invertible in $\Bbbk$. Hence by the above isomorphisms, $|A_{G_0}(x)|$ is invertible in $\Bbbk$. Therefore  any $\Bbbk[A_{G_0}(x)]$-module is a summand of the direct sum of copies of the regular representation, which again corresponds to $\pi_*\Bbbk_{\tilde{\mc{O}}}$.
\begin{align*}
\h^*_{G_0}(\mc{O},\pi_*\Bbbk_{\tilde{\mc{O}}})&\cong \h^*_{G_0}(\tilde{\mc{O}})\\ &\cong \h^*_{(G^x_0)^\circ}(pt)\text{  (by quotient equivalence)} \\ &\cong \h^*_{(G^x_0)^\circ-{red}}(pt).
\end{align*}
Here $(G^x_0)^\circ-red$ is the reductive quotient of $(G^x_0)^\circ$. 	By Lemma \ref{lm1.4} and assumption \ref{char}, $l$ is not a torsion prime for $(G^x)^\circ-red$. By \cite[Theorem ~2.44]{parity},  $\h^*_{(G^x_0)^\circ-{red}}(pt)$ will vanish in odd degrees if we can show that $l$ is not a torsion prime for $(G^x_0)^\circ-red$. This  will follow from  showing $(G^x_0)^\circ-{red}$ is a regular subgroup of $(G^x)^\circ-{red}$. 
Define a map $\psi: \mb{C}^\times \mapsto \mb{C}^\times \times G$ by,
\[t \mapsto (t^n,\chi(t)).\]Then $\mb{C}^\times\times G_0$ is the centralizer of $\psi(\mb{C}^\times)$. So $(\mb{C}^\times \times G_0)^x=(\mb{C}^\times \times G)^x\cap C_{\psi(\mb{C}^\times)}$, $C_{\psi(\mb{C}^\times)}$ is the centralizer of $\psi(\mb{C}^\times)$.  Therefore any maximal torus in $(\mb{C}^\times \times G)^x$ containing $\psi(\mb{C}^\times)$ commutes with $\psi(\mb{C}^\times)$, hence in $(\mb{C}^\times \times G_0)^x$. So   $(\mb{C}^\times \times G_0)^x$ is regular subgroup of   $(\mb{C}^\times \times G)^x$. Let us define $\mb{C}^\times \ltimes G^x$ by the action of $\mb{C}^\times$ on $G^x$ as $(t,g)\to \chi(t)g\chi(t^{-1})$. Now we define a map  $(\mb{C}^\times \times G)^x \to \mb{C}^\times \ltimes G^x$ by,
\[(t,g)\to [(t,\chi(t)g\chi(t^{-1})].\] It is easy to check this is an isomorphism and image of $\psi(\mb{C}^\times)$ under this isomorphism is contained in $1\ltimes G^x\cong G^x$. Similarly we have another isomorphism $ (\mb{C}^\times \times G_0)^x \cong \mb{C}^\times \ltimes G_0^x$. Therefore from the previous deduction we can say any maximal torus in   $G^x\subset 1 \ltimes G^x \subset \mb{C}^\times \ltimes G^x$ containing $\psi(\mb{C}^\times)$ that commute with $\psi(\mb{C}^\times)$ will be contained in $\mb{C}^\times \ltimes G_0^x$, so in $1\ltimes G_0^x\cong G_0^x$. Now we can conclude from the previous deduction that $G_0^x$ is regular subgroup of $G^x$. Hence $(G_0^x)^\circ-red$ is regular subgroup of $(G^x)^\circ-red$ and we are done.
\end{proof}


 \begin{theorem}\label{th8}
 	Let $(\mathcal{O},\mathcal{L})\in \mathscr{I}(\mf{g}_n)$ and  $\mf{l}_n$ and $\mathcal{L}'$ constructed above. Assume that $\mc{E}(\mc{O}_L,\mc{L}')$ exists, then \item
 	(a) The support of $\ind\mc{E}(\mc{O}_L,\mc{L}')$ is $\bar{\mathcal{O}}$, and
 	\item
 	(b)     $\ind\mc{E}(\mc{O}_L,\mc{L}')|_{ \mathcal{O}}=\mathcal{L}[\dim \mc{O}_L]$.
 \end{theorem}
\begin{proof}
\item
(a) Let $y\in \mf{g}_n$ be in the support of $\ind\mc{E}(\mc{O}_L,\mc{L}')$ . We need to show that $y\in \bar{\mathcal{O}}$. From the definition of induction, there exists $\eta \in \mf{p}_n$ and $g \in G_0$, such that,                                                                                                                                                                                                                                                                                                                                                                                                                                         $Ad(g)\eta=y$. Now both the support of $\ind\mc{E}(\mc{O}_L,\mc{L}')$ and $\bar{\mathcal{O}}$ are $G_0$-invariant. Hence we can replace $y$ by $\eta \in \mf{p}_n$. By \cite[ ~5.9]{Lu}, $\mf{p}_n$ coincides with the closure of the $P_0$-orbit of $x$ in $\mf{p}_n$ which is again contained in $\bar{\mathcal{O}}$. Hence $y \in \bar{\mathcal{O}}$ and part (a) is proved.
 \item (b) Recall the induction diagram
  \[\mf{l}_n \xleftarrow{\pi} \mf{p}_n \xrightarrow{e} G_0\times^{P_0}\mf{p}_n \xrightarrow{\mu} \mf{g}_n.\]
  Let $E_{\mathcal{O}}= \mu^{-1}(\mathcal{O})$. We  first show that $\mu$ is an isomorphism when restricted to $E_{\mathcal{O}}$. Actions of $G_0$ on $E_{\mathcal{O}}$ and $\mathcal{O}$ are compatible with the map $\mu$. Also action of $G_0$ on $\mathcal{O}$ is transitive. So to prove that $\mu$ is a bijection, it is enough to show that $\mu^{-1}(x)$ is a single point. Let $(g,\gamma) \in G_0\times \mf{p}_n$ be in the inverse image. So $Ad(g)\gamma= x$.  Therefore $x \in Ad(g)\mf{p}$ and by \cite[ ~5.7]{Lu}, $Ad(g)\mf{p}=\mf{p}$. Hence $g \in P_0$. Hence $(g,\gamma)=(1, Ad(g)\gamma)=(1,x)$. Hence $\mu^{-1}(x)$ is a singleton and $\mu$ is a bijection of smooth varieties, thus isomorphism on $E_{\mathcal{O}}$.
  
Let \begin{align*}
	\mathcal{G}= \ind\mc{E}(\mc{O}_L,\mc{L}')|_{\mathcal{O}}&=(\mu)_!(e^*\f^{G_0}_{P_0})^{-1}\pi^*(\mc{E}(\mc{O}_L,\mc{L}'))|_{\mathcal{O}}\\&=(\mu|_{E_{\mathcal{O}}})_{!}(e^*\f^{G_0}_{P_0})^{-1}\pi^*(\mc{E}(\mc{O}_L,\mc{L}'))|_{E_\mc{O}}.
\end{align*}
As $\mu|_{E_{\mathcal{O}}}$ is an isomorphism, hence $(\mu|_{E_{\mathcal{O}}})_!$ is an equivalence of categories. So in other words, $\mathcal{G}$ satisfies, \[(\mu|_{E_\mc{O}})^*(\mathcal{G})=(e^*\f^{G_0}_{P_0})^{-1}\pi^*(\mc{E}(\mc{O}_L,\mc{L}'))|_{E_\mc{O}}\] In fact $(\mu|_{E_\mc{O}})^*(\mathcal{G})$ is uniquely determined by,

\begin{equation}\label{6.33}
((e|_{E'_\mc{O}})^*\f^{G_0}_{P_0})(\mu|_{E_\mc{O}})^*(\mathcal{G})=(\pi|_{E'_\mc{O}})^*(\mc{E}(\mc{O}_L,\mc{L}')|_{\mc{O}_L})= (\pi|_{E'_\mc{O}})^*\mc{L}'[\dim \mc{O}_L],
\end{equation}
where $E'_{\mathcal{O}}=e^{-1}(E_{\mathcal{O}})$.
\[
\begin{tikzcd}
\bar{\mathcal{O}_L}                               & \bar{\mc{O}_L}+\mf{u}_n \arrow[l, "\pi"'] \arrow[r, "e"]                                & G_0\times^{P_0}(\bar{\mc{O}_L}+\mf{u}_n) \arrow[r, "\mu"]                      &\bar{\mc{O}}              \\
\mathcal{O}_L \arrow[u, hook]  & \mathcal{O}_L+\mf{u}_n \arrow[u, hook] \arrow[r, "e"] \arrow[l, "\pi"'] & G_0\times^{P_0}(\mathcal{O}_L+\mf{u}_n) \arrow[u, hook]  \\
\end{tikzcd}
\]As we already proved $E_\mc{O}\cong^\mu \mc{O}$ and $\mu$ is $G_0$ equivariant,  $E_\mc{O}$ is a single orbit. $G_0\times^{P_0}(\mc{O}_L+\mf{u}_n)$ is stable under $G_0$ action, therefore $E_\mc{O}\subset G_0\times^{P_0}(\mc{O}_L+\mf{u}_n)$ and it follows that $E'_\mc{O}\subset \mc{O}_L+\mf{u}_n$.  
Now we have the diagram below,
\[\begin{tikzcd}
	\mc{O}_L &E'_\mc{O}\arrow[l, "\pi|_{E'_\mc{O}}"] \arrow[r, hook, "e|_{E'_\mc{O}}"] & E_\mc{O} \arrow[r, "\mu|_{E_\mc{O}}"]& \mc{O}\\
	\mc{O}_L \arrow[u, "id"] \arrow[ur, hook, "j"'] \arrow[rrru, bend right, "i"]
\end{tikzcd}\]
 This diagram is commutative.
\begin{align*}
((e|_{E'_{\mc{O}}})^*\f^{G_0}_{P_0})(\mu|_{E_{\mc{O}}})^*\mc{L}[\dim \mc{O}_L]&=j_*i^*\mc{L}[\dim \mc{O}_L] \\&=(\pi|_{E'_{\mc{O}}})^*(\mc{L}[\dim\mc{O}_L]|_{\mc{O}_L})\\&= (\pi|_{E'_{\mc{O}}})^*(\mc{L}'[\dim \mc{O}_L])\\&=\pi^*\mc{E}(\mc{O}_L,\mc{L}')|_{E'_{\mc{O}}}.
\end{align*}
We already know $e|_{E'_\mc{O}}$ is closed embedding and $\mu|_{E_\mc{O}}$ is smooth, so the pull back of these maps induce faithful functor on the local systems. Hence from the above equation and (\ref{6.33}), we have $\mc{G}=\mc{L}[\dim \mc{O}_L]$. 

\end{proof}
\


 
\subsection{Modular reduction}\label{sec5.3} Recall from the background,  we assumed that there exists a finite extension $\mb{K}$ of $\mb{Q}_l$ with ring of integers $\mb{O}$ and residue field $\Bbbk$. Also assume for each $x\in \mc{N}_G$, all the irreducible representations of $A_G(x)$ are defined over $\mb{K}$. Let $K_G(\mc{N}_G,\Bbbk)$ denote the Grothendieck group generated by the isomorphism classes of simple objects in $\p_G(\mc{N}_G,\Bbbk)$. Similarly  define $K_G(\mc{N}_G,\mb{K})$. By \cite[ ~2.9]{J2}, there exists a $\mb{Z}$-linear map,
\[d: K_G(\mc{N}_G,\mb{K})\to K_G(\mc{N}_G,\Bbbk).\] This map is called  modular reduction map and is defined in the following way:
If $\mc{F}\in \p_G(\mc{N}_G,\mb{K})$ and $\mc{F}_\mb{O}$, a torsion free object in $\p_G(\mc{N}_G,\mb{O})$ such that $\mc{F}\cong \mb{K}\otimes_\mb{O} \mc{F}_\mb{O}$, then
\[d([\mc{F}])=[\Bbbk\otimes _\mb{O}^L \mc{F}_\mb{O}].\]If $\mc{F}=\mc{IC}(C,\mc{E})$, then there exists a $G$-equivariant local system $\mc{E}_\mb{O}$ on $C$ such that $\mc{E}=\mb{K}\otimes_\mb{O}\mc{E}_\mb{O} $, and
\[d[\mc{IC}(C,\mc{E})]=[\Bbbk \otimes^L_\mb{O} \mc{IC}(C,\mc{E}_\mb{O})].\]By \cite[prop ~2.39]{parity}, $ \Bbbk \otimes^L_\mb{O} \mc{E}(C,\mc{E}_\mb{O})\cong \mc{E}(C,\Bbbk\otimes^L_\mb{O}\mc{E}_\mb{O}) $.  We can call $ \Bbbk\otimes^L_\mb{O}\mc{E}_\mb{O} $ to be $\mc{E}_\Bbbk$.
\begin{theorem}\label{th6.3}
The modular reduction above gives a well-defined map with the following properties:
\begin{enumerate}
	\item $Irr(\mb{K}[G^x/(G^x)^\circ]-mod)\xrightarrow{\cong} Irr(\Bbbk[G^x/(G^x)^\circ]-mod)$.
	\item If $M$ is a torsion free module in $\mb{O}[G^x/(G^x)^\circ]$, then the direct summands of $\mb{K}\otimes_\mb{O} M$ are in bijection with  the direct summands of $\Bbbk\otimes_\mb{O} M$.
\end{enumerate}
\end{theorem}
The proof follows from \cite[Theorem ~82.1]{CR}.
\\
\\
Let $P$ be a parabolic subgroup with $L$ as the Levi factor.
\begin{theorem}\label{th6.1}
If $(C,\mc{F})\in \ms{I}(L)^{0-\cu}$ and  $(C,\mc{G})\in \ms{I}(L,\mb{K})^{0-\cu}$, whose modular reduction is $\mc{F}$, then,
\begin{enumerate}
\item $\mc{IC}(C,\mc{G}_\mb{O})$ is clean, and
\item $\Ind^G_P\mc{IC}(C,\mc{G}_\mb{O})$ is parity.	
\end{enumerate}
	
\end{theorem}
\begin{proof}
\begin{enumerate}
\item If there exists $y\in \bar{C}-C$ such that $\mc{IC}(C,\mc{G}_\mb{O})_y\ne 0$, then it must have torsion part only; otherwise, $\mb{K}\otimes^L\mc{IC}(C,\mc{G}_\mb{O})_y\ne 0$ but  is same as $\mc{IC}(C,\mc{G})_y$ which is zero by Conjecture \ref{con}. Also by \cite[~ 2.6]{J2}, $\Bbbk\otimes^L_{\mb{O}} \mc{IC}(C,\mc{G}_\mb{O})$ is a perverse sheaf with  $\Bbbk\otimes^L_{\mb{O}} \mc{IC}(C,\mc{G}_\mb{O})|_C=\Bbbk \otimes^L_\mb{O} \mc{G}_\mb{O}[\dim C]$. We have the open and closed embeddings below,
\[\begin{tikzcd}
C \arrow[r, hook, "j"] & \bar{C} & \bar{C}-C\arrow[l, hook, "i"']	\end{tikzcd},
 \]which gives rise to the distinguished triangle,
 \[j_! \Bbbk \otimes^L_\mb{O} \mc{G}_\mb{O}[\dim C]\to \Bbbk\otimes^L_{\mb{O}} \mc{IC}(C,\mc{G}_\mb{O}) \to i_*i^* (\Bbbk\otimes^L_{\mb{O}} \mc{IC}(C,\mc{G}_\mb{O}))\to\hspace{2mm}.\]
 By conjecture \ref{con}, $ j_! (\Bbbk \otimes^L_\mb{O} \mc{G}_\mb{O}[\dim C] )=\mc{IC}(C,\Bbbk\otimes^L_\mb{O} \mc{G}_\mb{O})$. 
 The third morphism in the above distinguished triangle is $0$, as $i^!(\mc{IC}(C, \Bbbk\otimes^L_\mb{O} \mc{G}_\mb{O}))=0$ by conjecture \ref{con}. Therefore this distinguished triangle splits and we have,
 \[ \Bbbk\otimes^L_{\mb{O}} \mc{IC}(C,\mc{G}_\mb{O})= \mc{IC}(C,\Bbbk\otimes^L_\mb{O} \mc{G}_\mb{O})\oplus i_*i^* (\Bbbk\otimes^L_{\mb{O}} \mc{IC}(C,\mc{G}_\mb{O})).\]   As $  \Bbbk\otimes^L_{\mb{O}} \mc{IC}(C,\mc{G}_\mb{O}) $ and $ \mc{IC}(C,\Bbbk\otimes^L_\mb{O} \mc{G}_\mb{O}) $ are perverse, so $ i_*i^* (\Bbbk\otimes^L_{\mb{O}} \mc{IC}(C,\mc{G}_\mb{O})) $ must be perverse.
  Now it will not be hard to check that for a torsion $\mb{O}$-module $M$, $\h^i(\Bbbk \otimes^L_{\mb{O}} M)$ is nonzero for $i=0,-1$. So if we choose an open orbit $\mc{O}'$ in the support of $i_*i^*(\Bbbk\otimes\mc{IC}(C,\mc{G}_\mb{O}))$ such that $y$ is in that orbit then by the above statement we will have $\h^i(\Bbbk \otimes^L_{\mb{O}} \mc{IC}(C,\mc{G}_\mb{O})_y)\ne 0$ for $i=-\dim \mc{O}',-\dim{O}'-1$, which contradicts the perversity of $ i_*i^* (\Bbbk\otimes^L_{\mb{O}} \mc{IC}(C,\mc{G}_\mb{O}))$.
\item Note that $\mb{K}\otimes^L_\mb{O}\Ind^G_P\mc{IC}(C,\mc{G}_\mb{O})=\Ind^G_P\mc{IC}(C,\mc{G})$ which is parity because it is in characteristic $0$.  Also, $\Bbbk\otimes^L_\mb{O}\Ind^G_P\mc{IC}(C,\mc{G}_\mb{O})\cong \Ind^G_P\mc{E}(C,\mc{F})$ 
is parity by Conjecture \ref{2.6(c)}. Combining these two facts and using \cite[Prop. ~2.37]{parity}, $\Ind^G_P\mc{IC}(C,\mc{G}_\mb{O})$ is parity.

\end{enumerate}

\end{proof}

\begin{theorem}\label{th6.2}
For $(C,\mc{F})\in \ms{I}(G)$, there exists a Levi subgroup $L$ and a pair $(C',\mc{F}')\in \ms{I}(L)^{0-\cu}$ such that, $\h^{-\dim{C}}(\Ind^G_P\mc{IC}(C',\mc{F}'))|_C$ contains $\mc{F}$ as a direct summand.
	
\end{theorem}
\begin{proof} By \cite{AJHR}, if $l$ is rather good, which is our assumption here, we have a bijection,
 \[K_G(\mc{N}_G,\mb{K})\xrightarrow{\cong} K_G(\mc{N}_G,\Bbbk).\]
Let $(C,\mc{G})\in \ms{I}(G,\mb{K})$ be the pair whose modular reduction is $\mc{F}$.
By \cite{Lu3}, there exists a parabolic subgroup $P$ with Levi $L$ and $(C',\mc{G}')\in \ms{I}(L,\mb{K})^{0-\cu}$ be such that $\mc{IC}(C,\mc{G})$ is a direct summand of  $\Ind^G_P\mc{IC}(C',\mc{G}')$ .
\[\mb{K}\otimes^L \h^{-\dim{C}}(\Ind^G_P\mc{IC}(C',\mc{G}'_\mb{O}))|_C \cong\h^{-\dim{C}}(\Ind^G_P\mc{IC}(C',\mc{G}'))|_C\] which contains $\mc{G}$. Now, $\h^{-\dim{C}}(\Ind^G_P\mc{IC}(C',\mc{G}'_\mb{O}))|_C $ is torsion-free. If not, then  $ \Bbbk \otimes^L_\mb{O}\h^{-\dim{C}}(\Ind^G_P\mc{IC}(C',\mc{G}'_\mb{O}))|_C $ has cohomology concentrated in two consecutive degrees, which contradicts Theorem \ref{th6.1}(2). By Theorem \ref{th6.3}, direct summands appearing in $\h^{-\dim{C}}(\Ind^G_P\mc{IC}(C',\mc{G}'))|_C$     are in bijection with direct summands appearing in $\h^{-\dim{C}}(\Ind^G_P\mc{IC}(C',\mc{G}'_\Bbbk))|_C$. Set $\mc{F}'$ to be $\mc{G}'_\Bbbk$. Hence $\mc{F}'\in \ms{I}(L)^{0-\cu}$ and 
 $\h^{-\dim{C}}(\Ind^G_P\mc{IC}(C',\mc{F}'))|_C$ contains $\mc{F}$ as a direct summand.

\end{proof}                                                                                              

\begin{proposition}\label{prop6.6.1}
Let $(\mc{O},\mc{L})\in \ms{I}(\mf{g}_n)$.	There exists an  integer $b$ and $(\mc{O}',\mc{L}')\in \i'$ for some parabolic subgroup $P$ with the Levi subgroup $L$, such that $\mc{L}$ is a direct summand of $\h^{b-\dim \mc{O}'}(\ind\mc{IC}(\mc{O}',\mc{L}'))|_\mc{O}$.  
\end{proposition}
\begin{proof}

Let $(\mc{O},\mc{L})\in \ms{I}(\mf{g}_n)$ and $(C,\mc{E})\in \ms{I}(G)$ such that $C\cap \mf{g}_n=\mc{O}$ and $\mc{E}|_\mc{O}=\mc{L}$. By Theorem \ref{th6.2}, there exists  $Q$, a parabolic subgroup containing $ \chi(\mb{C}^\times)$  with $M$, the Levi subgroup such that $(C',\mc{E}')\in \ms{I}(M)^{0-\cu}$ and 	$\h^{-\dim C}(\inn^G_Q\mc{IC}(C',\mc{E}'))|_C$ contains $\mc{E}$ as direct summand. 

For this we will  imitate the proof of Theorem \ref{th5}. Let $y\in \mc{O}$, we construct the parabolic induction diagram corresponding to $(Q,M,C')$. That is,
\[C'\xleftarrow{a} C'+\mf{u}_Q\xrightarrow{b} G\times^Q (C'+\mf{u}_Q) \xrightarrow {c} \mc{N}_G\]Let $Y_y=c^{-1}(y)$ and  recall that, 

 \[\h^a_c(Y_y,(b^*\f^G_Q)^{-1}a^* \mc{E}'[\dim C']|_{Y_y})=0\text{ for $a$ odd.}\]

Also recall the action of $\mb{C}^\times$ on $Y_y$ defined in  the proof of Theorem \ref{th5}. Let again $Y_y^{\mb{C}^\times}$ be the fixed point set of this action. It is not hard to see that the stabilizer of each point in the complement of $(Y_y)^{\mb{C}^\times}$ is trivial. Hence  by The Lemma \ref{lm7}, $\dim \h^*_{\mb{C}^\times}(Y_y-(Y_y)^{\mb{C}^\times})< \infty$. Now  using Lemma \ref{euler}, the Euler characteristic of $\h^*_c(Y_y-(Y_y)^{\mb{C}^\times})$ is $0$, so we have, 

\begin{equation}\label{28}
\sum_{a}{(-1)^a\h^a_c(Y_y,(b^*\f^G_Q)^{-1}a^* \mc{E}'[\dim C']|_{Y_y})}\\=\sum_{a}{(-1)^a\h^a_c((Y_y)^{\mb{C}^\times}, (b^*\f^G_Q)^{-1}a^* \mc{E}'[\dim C']|_{Y_y^{\mb{C}^\times}})}
\end{equation}
As the stabilizer of each point in $Y_y-(Y_y)^{\mb{C}^\times}$ is trivial, so from Theorem \ref{th7},
  \[\h^a_c(Y_y^{\mb{C}^\times},(b^*\f^G_Q)^{-1}a^* \mc{E}'[\dim C']|_{Y_y^{\mb{C}^\times}})=0\text{ for $a$ odd.}\]

Combining both the result we have,
\begin{equation}\label{29'}
\sum_{a \text{ even}}\h^a_c(Y_y,(b^*{\f^G_Q})^{-1}a^* \mc{E}'[\dim C']|_{Y_y})=\sum_{a\text{ even}}\h^a_c(Y_y)^{\mb{C}^\times},(b^*\f^G_Q)^{-1}a^* \mc{E}'[\dim C']|_{(Y_y)^{\mb{C}^\times}})
\end{equation}
Now let $Q^i$'s denote the $G_0$-orbits of $Q$ in the set of all parabolic subgroups containing $\chi(\mb{C}^\times)$, for $i=1,...,b$. Define $Z^i$ as before but in terms of $Q^i$ and $C'^i$, where an element of $G$ conjugates $Q$ to $Q^i$, conjugating $C'$ by the same element gives $C'^i$ contained in $\mf{m}^i$.

\[Z^i=\{g(Q^i)_0\in G_0/(Q^i)_0| Ad(g^{-1})y\in (\pi^i)^{-1}(C')^i\}.\]
Then as before we get,
\[Y_y^{\mb{C}^\times}=\sqcup_i Z^i,\]and, 
\begin{equation}
	\h^a_c(Z^i,(b^*\f^G_Q)^{-1}a^* \mc{E}'[\dim C']|_{Z^i}  )=\h^a(\inn^\mf{g}_{\mf{q}^i}(\mc{IC}(\mathcal{O}^i,\mathcal{L}^i)[\dim C'-\dim \mc{O}^i])_y,
\end{equation} 
where $(\mc{O}^i,\mc{L}^i)\in \ms{I}(M^i)$. These  come from conjugating and restricting $(C',\mc{E}')$. Now combining equation \ref{29'} and the above result we have,\[ \sum_{a \text{ even }}\h^a_c(Y_y,(b^*\f^G_Q)^{-1}a^* \mc{E}'[\dim C']|_{Y_y})=   \sum_{i, a \text{ even}}\h^a(\inn^\mf{g}_{\mf{q}^i}(\mc{IC}(\mathcal{O}^i,\mathcal{L}^i)[\dim C'-\dim \mc{O}^i])_y.\]
Now $\h^a_c(Y_y,(b^*\f^G_Q)^{-1}a^* \mc{E}'[\dim C']|_{Y_y})$ is the same as $\h^a(\inn^G_Q\mc{IC}(C',\mc{E}'))_y$. By assumption $\mc{E}$ occurs as direct summand of  $\h^{-\dim{C}}(\Ind^G_Q\mc{IC}(C',\mc{E}'))|_C$. As $\mc{L}=\mc{E}|_\mc{O}$, therefore for some $i$ and  $a$ even,  $\mc{L}$ should appear as a direct summand of $\h^a(\inn^\mf{g}_{\mf{q}^i}(\mc{IC}(\mathcal{O}^i,\mathcal{L}^i)[\dim C'-\dim \mc{O}^i])_y$. We  call this $Q^i$ to be $P$,  $M^i$ to be $L$ and $(\mc{O}^i,\mc{L}^i)$ to be $(\mc{O}',\mc{L}')$. Hence we get the desired result that $\mc{L}$ is a direct summand of $\h^{a+\dim C'-\dim \mc{O}'}(\ind \mc{IC}(\mc{O}',\mc{L}'))$. we can call $a+\dim C'$ to be $b$.
\end{proof}

\subsection{Normal complexes}

\begin{definition}
 An $A\in D^b_{G_0}(\mf{g}_n)$ is called normal if there exists $(\mc{O},\mc{L})\in \i'$ 
	 such that some shift of $A$ is a direct summand of  $\ind(\mathcal{E} (\mathcal{O},\mathcal{L}) )$.
\end{definition}
\begin{theorem}\label{th9}
For $(\mathcal{O},\mathcal{L})\in \mathscr{I}(\mf{g}_n)$, $\mathcal{E} (\mathcal{O},\mathcal{L})$ exists and is a normal complex.	
\end{theorem}
\begin{proof}
For $(\mc{O},\mc{L})\in \ms{I}(\mf{g}_n)$, we can construct the Levi subgroup $L$ as in \ref{subsec6.1} and $(\mc{O}_L,\mc{L}')\in \ms{I}(\mf{l}_n)$ as  in \ref{subsec6.2} and by Theorem \ref{th8}, $\ind \mc{E}(\mc{O}_L, \mc{L}')|_\mc{O}=\mc{L}[\dim \mc{O}_L]$.	
Now by Proposition \ref{prop6.6.1}, there exists a parabolic subgroup $Q\subset L$ and a Levi subgroup $M\subset Q$ with $(\mc{O}',\mc{L}'')\in \ms{I}(\mf{m}_n)^{0-\cu}$ such that $\h^{b-\dim \mc{O}'}(\inn^\mf{l}_\mf{q}\mc{IC}(\mc{O}',\mc{L}'')|_{\mc{O}_L}$ contains $\mc{L}'$ as a direct summand. By Theorem \ref{th5}, $\inn^\mf{l}_\mf{q}\mc{IC}(\mc{O}',\mc{L}'')$ is parity. By  proposition \ref{nrigid}, $\mc{O}_L$ is open in $\mf{l}_n$. Combining these two above facts, we can see that $\mc{E}(\mc{O}_L,\mc{L}')$ exists and  is direct summand of $\inn^\mf{l}_\mf{q}\mc{IC}(\mc{O}',\mc{L}'')$.  Using  the fact that induction is transitive it follows, $\mc{L}[\dim \mc{O}_L]$ is direct summand of  $\ind\mc{IC}(\mc{O}',\mc{L}'')|_\mc{O}$. By Theorem \ref{th8}, support of $\ind\mc{E}(\mc{O}_L,\mc{L}')$ is $\bar{\mc{O}}$. Therefore $\mc{E}(\mc{O},\mc{L})$ exists and is direct summand of $\ind\mc{E}(\mc{O}',\mc{L}'')$.
\end{proof}

\subsection{Induction preserves parity}\label{sec8}

\begin{theorem}Let $P$ be a parabolic subgroup of $G$ with a Levi factor $L$.
For any pair $(\mathcal{O}, \mathcal{L}) \in \mathscr{I}(\mf{l}_n)$, the induction functor sends parity complexes to parity complexes.	
\end{theorem}

\begin{proof}
By Theorem \ref{th9}, there exist a cuspidal pair $(\mathcal{C}, \mathcal{F})\in \ms{I}(\mf{m}_n)^{0-\cu}$ , where $M$ is the Levi subgroup of $L$ such that \[\inn^\mf{l}_{\mf{l}\cap \mf{p}}(\mathcal{E}(\mathcal{C},\mathcal{F}))=\mathcal{E}(\mathcal{O},\mathcal{L})[k]\oplus ... \text{, for some } k \in \mb{Z}.\] If we apply $\ind$ on  both sides and use the transitivity of induction, then we get, \[\ind(\mathcal{E}(\mathcal{C},\mathcal{F}))=\ind(\mathcal{E}(\mathcal{O},\mathcal{L}))[k]\oplus ...\]
Now the left-hand side is parity by Theorem \ref{th5}. Hence so is the right-hand side. So induction preserves the parity of $\mathcal{E}(\mathcal{O}, \mathcal{L})$.
\end{proof}

\section{Examples}\label{sec10}

\subsection{Several cases for conjecture \ref{2.6(c)} }\label{sec9}
Here, we want to calculate $\Ind$ for some classical algebraic groups and also  we want to show that conjecture \ref{2.6(c)} works well for cuspidal pairs on the Levis.

\begin{example}\textbf{$\mf{sp}_4$-case:}
Let $G=Sp_4$. The symplectic form is defined by the matrix
$B=\begin{pmatrix}
0&0&1&0\\0&0&0&1\\-1&0&0&0\\0&-1&0&0	
\end{pmatrix}$ 
and the inner product is  $Q(v,w)=v^t B w$, where $v,w\in \mb{C}^4$. $Sp_4$ is defined as the group of automorphisms $A$ from $\mb{C}^4$ to $\mb{C}^4$, such that $Q(Av,Aw)=Q(v,w)$. Here a torus is of the form $diag(t_1,t_2,t_1^{-1},t_2^{-2})$. Let $\{e_1,e_2,e_3,e_4\}$ be the standard basis of $\mb{C}^4$. Hence the root system is $\Phi=\{ \pm e_1\pm e_2, \pm2e_1,\pm2e_2\}$. 
The orthogonal complement of a vector space $V$, denoted by $V^\perp$, is the set of all vectors  having inner product $0$ with all the vectors in $V$. Clearly the orthogonal complement of $\langle e_1 \rangle$ is $\langle e_1,e_2,e_4\rangle$.
\subsubsection{Levi subgroups}
The Levi subgroups are $T,GL(2), GL(1)\times Sp(2)$, up-to conjugacy.
According to our  assumption for the characteristic $l$,  $l\ne2$.
 So from \cite{Lu3},
  we can see that $T$ and $GL(1)\times Sp(2)$ have cuspidal pairs. For $T$ the parity condition has been checked in \cite[~ 4.3]{parity}. In the  case of $GL(1)\times Sp(2)$, the cuspidal pair is of the form $(\mc{O}_{prin},\mc{L})$, where $\mc{O}_{prin}$ is the $Sp_2$-principal orbit in $\mf{sp}(2)=\mf{sl}_2$ and $\mc{L}$ is the nontrivial $SL(2)$-equivariant local system on $\mc{O}_{prin}$.
The Levi $GL(1)\times Sp(2)$ comes from the root $\al=2e_2$. The Levi and  nilpotent  subalgebras related to the root $\al=2e_2$ are of the form
\[\mf{l}=\{\begin{pmatrix}
	a&0&0&0\\0&b&0&c\\0&0&-a&0\\0&d&0&-b\\
\end{pmatrix}| \hspace{2mm}a,b,c,d\in \mb{C}\}\cong \mf{sl}_2\times \mb{C},
\]and
\[\mf{u}_P=\{\begin{pmatrix}
0&x&z&y\\0&0&y&0\\0&0&0&0\\0&0&-x&0\\	
\end{pmatrix}|\hspace{2mm}x,y,z \in\mb{C}\}.
\]Hence, $\mf{p}=\{\begin{pmatrix}
	a&x&z&y\\0&b&y&c\\0&0&-a&0\\0&d&-x&-b\\
\end{pmatrix}|\hspace{2mm}a,b,c,d,x,y,z\in\mb{C}\}
$.
 Now we want to calculate $\Ind$ with respect to these parabolic and Levi. 
 Recall the parabolic induction diagram,
 \[\mc{N}_L\xleftarrow{\pi} \mc{N}_L+\mf{u}_P\xrightarrow{e}G\times^P(\mc{N}_L+\mf{u}_P)\xrightarrow{\mu}\mc{N}_G.\]   
The crucial step is to calculate the push forward of the map $\mu$. We can interpret the space $G\times^P(\mc{N}_L+\mf{u}_P)$ in a different way,
\begin{align*}
G\times^P{(\mc{N}_L+\mf{u}_P)}&\cong \{(gP,x)\in G/P\times \mc{N}_G| Ad(g^{-1})x \in \mc{N}_L+\mf{u}_P\}\text{, by  }(g,x)\to (gP,Ad(g)x).\\ &\cong \{(gP,x)\in G/P\times \mc{N}_G |Ad(g^{-1})x \in Lie(P)\}\text{, as $x\in \mc{N}_G$,  $Ad(g^{-1})x\in \mc{N}_G$.}
\end{align*}
 $Ad(g^{-1})x\in \mf{p}$ means it preserves the partial flag $\langle e_1\rangle\subset\langle e_1,e_2,e_4\rangle$,  call it $E$. This implies  that $x$ preserves $gE$. Hence the definition becomes
 
\begin{align}
 G\times^P(\mc{N}_L+\mf{u}_P)=\left\{(V_1\subset V_3,x) \middle |\
 \begin{aligned}
& V_1\subset V_3\text{ is a partial flag of dimension $1$ and $3$},\\
& V_1^\perp=V_3, x\in \mc{N}_G \text{ preserves }V_3 \text{ and }V_1
\end{aligned} \right\}.
\end{align}

 Now the map $\mu$ becomes  projection on the second coordinate.
 By \cite[~5.2]{CM}, we can find the orbits in $\mf{sp}_4$, which are $\mc{O}[4],\mc{O}[2,1^2],\mc{O}[2^2],\mc{O}[1^4]$. The representatives from these orbits are,
 \[
 \begin{pmatrix}
	0&1&0&0\\0&0&0&1\\0&0&0&0\\0&0&-1&0\\
\end{pmatrix}, \begin{pmatrix}
0&0&1&0\\0&0&0&0\\0&0&0&0\\0&0&0&0\\	
\end{pmatrix},
\begin{pmatrix}
0&0&1&0\\0&0&0&1\\0&0&0&0\\0&0&0&0\\	
\end{pmatrix}, \{0\},\]
respectively. Now we are interested in the fibers of the representatives of each orbits. For each orbit as above,  we call $x$ to be the representative.
\begin{lemma}
For $x$ defined above, $\mu^{-1}(x)$ has the following descriptions,
\begin{enumerate}
\item for $x\in \mc{O}[2,1^2]$, 	 $\mu^{-1}(x)\cong \mb{P}^2$,
\item for $x\in \mc{O}[2^2]$,  $\mu^{-1}(x)\cong \mb{P}^1$,
\item for $x\in \mc{O}[4]$,  $\mu^{-1}(x)\cong \{pt\}$,
\item for $x\in \mc{O}[1^4]$,  $\mu^{-1}(x)\cong G/P$.
\end{enumerate}
	
\end{lemma}
\begin{proof}
\begin{enumerate}
	\item If $x\in \mc{O}[2,1^2]$, then $\ker{x}=\langle e_1,e_2,e_4\rangle$. 
	 For a flag $V_1\subset V_3$ in $G\times^P(\mc{N}_L+\mf{u}_P)$,  $x$ stables both $V_3$ and $V_1$. So if $\langle v\rangle=V_1$, then ${v}$ should either go to $0$  or to some scalar multiplication of $v$ under  the map $x$. Let $v=ae_1+be_2+ce_3+de_4$, then $x.v=ce_1$. If $c=0$, then $x.v=0$ which means  $v\in ker{x}$. If $c\ne 0$ then $x.v\in V_1$, so $b=c=d=0$, a contradiction. Hence $V_1\subset \ker{x}$ which is $3$ dimensional. Once we choose $V_1$, $V_3$ is automatically determined by the condition $V_1^\perp=V_3$. Hence $\mu^{-1}(x)\cong \mb{P}^2$.

\item
Now for $\mc{O}[2^2]$, to find the fiber  we will proceed as before. For $\mc{O}[2^2]$, $\ker{x}=\langle e_1,e_2\rangle$. If $v$ is the generator of $V_1$ then $x.v$ is either $0$ or in $\langle v \rangle$. Now if $v=ae_1+be_2+ce_3+de_4$, then $x.v=ce_1+de_2$. Hence either $c=d=0$  implying $v\in \ker{x}$, otherwise, $ce_1+de_2=\lambda v$  implying  $c=d=0$, a contradiction. Therefore, $V_1\subset \ker{x}$ which is two-dimensional. Hence in this case $\mu^{-1}(x)\cong \mb{P}^1$.
\item
For $\mc{O}[4]$, we can check that $\langle e_1\rangle\subset \langle e_2,e_1,e_4\rangle $ is the only flag which satisfies all the conditions to be in the inverse image, so the fiber is just a point. 
\item
 For $\mc{O}[1^4]$ the fiber is the whole space $G/P$.

\end{enumerate}
\end{proof}
We can find the dimension and the fundamental groups of the orbits from \cite{CM}. Now as $L\cong SL_2\times \mb{C}^\times$, hence the orbits in $L$ are $\mc{O}_{prin}$ and $\{0\}$. We want to calculate  $\mu^{-1}(x)\cap G\times^P(\mc{O}_0+\mf{u}_P)$ and $\mu^{-1}(x)\cap G\times^P(\mc{O}_{prin}+\mf{u}_P)$ for each representative $x$. 

\begin{lemma}\label{lm12}
	Let $x\in \mc{N}_L+\mf{u}_P$. Then,
	\begin{enumerate}
		\item $x\in \mc{O}_{prin}+\mf{u}_P$ if and only if the action of $x$ on $\langle e_1,e_2,e_4\rangle/\langle e_1\rangle$ is nonzero,
		\item $x\in \mc{O}_0+\mf{u}_P$ if and only if the action of $x$ on $\langle e_1,e_2,e_4\rangle/\langle e_1\rangle$ is zero.
	\end{enumerate}
\end{lemma}
The proof follows from some simple matrix calculations.
	

Thus from the above lemma and the definition that we gave in the beginning, \[G\times^P(\mc{O}_0+\mf{u}_P)=\left\{(gP,x)\in G/P\times \mc{N}_G\middle| \hspace{2mm}
\begin{aligned}
& Ad(g^{-1})x\text{ preserves the flag }\langle e_1 \rangle \subset \langle e_1,e_2,e_3\rangle \\ & \text{ and }Ad(g^{-1})x\text{ is zero on  }\langle e_1,e_2,e_3\rangle/\langle e_1 \rangle
\end{aligned}\right\}.\]Which is same as,

\[G\times^P(\mc{O}_0+\mf{u}_P)=\left\{(gP,x)\in G/P\times \mc{N}_G\middle| \hspace{2mm}
\begin{aligned}
& x\text{ preserves the flag }g.\langle e_1 \rangle \subset g.\langle e_1,e_2,e_3\rangle \\ & \text{ and }x\text{ is zero on  }g.\langle e_1,e_2,e_3\rangle/g.\langle e_1 \rangle
\end{aligned}\right\}.\]
Which is again same as,\[
\left\{(V_1\subset V_3,x) \middle |\
 \begin{aligned}
& V_1\subset V_3\text{ is a partial flag of dimension $1$ and $3$},\\
& V_1^\perp=V_3, x\in \mc{N}_G \text{ preserves }V_3 \text{ and }V_1\text{ with }x \text{ is }0\text{ on }V_3/V_1 
\end{aligned} \right\}.
\]

Therefore for each representative $x$,  \[\mu^{-1}(x)\cap G\times^P(\mc{O}_0+\mf{u}_P)= \left\{V_1\subset V_3 \middle |\
 \begin{aligned}
& V_1\subset V_3\text{ is a partial flag of dimension $1$ and $3$},\\
& V_1^\perp=V_3, x \text{ preserves }V_3 \text{ and }V_1\text{ with }x \text{ is }0\text{ on }V_3/V_1 
\end{aligned} \right\}.
\] 

\begin{lemma}
For $x$ being the representative of each orbit, $\mu^{-1}(x)\cap G\times^P(\mc{O}_0+\mf{u}_P)$ satisfies the fifth column in  the table given below.
\end{lemma}
\begin{proof}
\begin{enumerate}

\item
For $\mc{O}[1^4]$, it is not hard to see, $\mu^{-1}(x)\cap G\times^P(\mc{O}_0+\mf{u}_P)=G/P$.

\item For $\mc{O}[2,1^2]$,
 we have already  seen that if $\langle v_1\rangle\subset \langle v_1,v_2,v_3\rangle$ is in $\mu^{-1}(x)$, then $v_1\in \ker{x}=\langle e_1,e_2,e_4\rangle$. As $x$ is zero on $\langle v_1,v_2,v_3\rangle/\langle v_1\rangle$ that means  $v_2$ and $v_3$ should either go to $0$ or to $\langle v_1\rangle$. If both go to $0$, that is both are in the kernel, then $\langle v_1,v_2,v_3\rangle=\langle e_1,e_2,e_4\rangle$. Using the condition $V_1^\perp=V_3$, definitely $V_1=\langle e_1 \rangle$. 
 If one of them goes to $v_1$, let's say $v_2$. But then $x.v_2$ is some scalar multiplication of $e_1$ as $Im(x)=\langle e_1 \rangle$. This implies $\langle v_1 \rangle=\langle e_1 \rangle$. Again using the fact that $\langle v_1\rangle^\perp= \langle v_1,v_2,v_3\rangle$, we can see,  $\langle v_1,v_2,v_3\rangle=\langle e_1,e_2,e_4\rangle$. Hence $\mu^{-1}(x)\cap G\times^P(\mc{O}_0+\mf{u}_P)$ is a single point $\{\langle e_1 \rangle \subset \langle e_1,e_2,e_4 \rangle\}$.
\item
Now for $\mc{O}[2^2]$, again $\langle v_1 \rangle \subset \ker{x}$. But here $\ker{x}=\{e_1,e_2\}$. This means $v_1$ is of the form  $ae_1+be_2$. If $a=0$ that means  $\langle v_1 \rangle=\langle e_2 \rangle $ and $v_1^\perp=\langle e_2,e_1,e_3\rangle$. But $e_3$ goes to $e_1$ under $x$, so the map does not induce a zero map on the quotient  $\langle v_1,v_2,v_3 \rangle / \langle v_1 \rangle$. Similarly if $b=0$ then   $\langle v_1 \rangle=\langle e_1 \rangle $, therefore  $v_1^\perp=\langle e_2,e_1,e_4\rangle$. But $e_4$ goes to $e_2$ under $x$, so the map does not induce a $0$ map on the quotient  $\langle v_1,v_2,v_3 \rangle / \langle v_1 \rangle$. Now we consider the case where both $a$ and $b$ are non-zero.    We have  $\langle v_1 \rangle^\perp=\langle e_1 , e_2, be_3-ae_4 \rangle$. Now under the map $x$, $e_1,e_2$ both goes to $0$ but $be_3-ae_4$ goes to $be_1-ae_2$. The action of $x$ should be zero on the quotient, that means  $be_1-ae_2$ must be a scalar multiplication of $v_1$. This implies $b^2+a^2=0 $ or $a=\pm ib$. Therefore the flags that satisfy all the conditions to be in $\mu^{-1}(x)\cap G\times^P(\mc{O}_0+\mf{u}_P)$ are $\langle ie_1+e_2 \rangle \subset \langle e_1,e_2,e_3-ie_4 \rangle$ and  $\langle -ie_1+e_2 \rangle \subset \langle e_1,e_2,e_3+ie_4 \rangle$.

\item For $\mc{O}[4]$, $\langle e_1\rangle\subset \langle e_2,e_1,e_4\rangle$ is the only flag in the inverse image but it does not satisfy this condition hence $\mu^{-1}(x)\cap G\times^P(\mc{O}_0+\mf{u}_P)=\emptyset$.

\end{enumerate}
\end{proof}

\FloatBarrier
\vspace{\baselineskip}
\begin{table}[h!]
\caption{Orbits in $\mf{sp}_4$}
\begin{center}
\begin{tabular}{|c|c|c|c|c|c|}
\hline
orbits:& $\mc{O}[4]$ & $\mc{O}[2^2]$ & $\mc{O}[2,1^2]$ & $\mc{O}[1^4]$ \\ \hline
$\dim:$ & 8&6&4&0\\
\hline
$\pi_1:$& $\mb{Z}/2$ &    $\mb{Z}/2$ &   $\mb{Z}/2$ & 0\\
\hline
$\mu^{-1}(x): $& $\{pt\}$& $\mb{P}^1$& $\mb{P}^2$ & $G/P$\\
\hline
$\mu^{-1}(x)\cap G\times^P(\mc{O}_0+\mf{u}_P):$&$\emptyset$&$\{pt\}\sqcup \{pt\}$&$\{pt\}$&$G/P$\\
\hline
$\mu^{-1}(x)\cap G\times^P(\mc{O}_{prin} +\mf{u}_P):$ & $\{pt\} $& $\mb{A}^1-\{pt\}$ & $\mb{P}^2-\{pt\}$ & $\emptyset$ 
\\
\hline
\end{tabular}
\end{center}
\end{table}
\vspace{\baselineskip}
\FloatBarrier

We are now ready to calculate $\Ind^G_P$ for cuspidal pairs. 
 The only cuspidal pair on $SL_2$ is $(\mc{O}_{prin},\mc{L})$, where $\mc{L}$ is the nontrivial local system on $\mc{O}_{prin}$.
 Now recall the parabolic induction diagram for cuspidal pair defined in \ref{subsec5.2}. As $(\mc{O}_{prin},\mc{L})$ is cuspidal, so $\inn^G_P\mc{IC}(\mc{O}_{prin},\mc{L})=c_!(b^*\f^G_P)^{-1}a^*\mc{L}[\dim \mc{O}_{prin}]= c_!(b^*\f^G_P)^{-1}a^*\mc{L} [2]$. As we know pull-back of some local system is again a local system, thus   $(b^*\f^G_P)^{-1}a^*\mc{L} [2]$ is a local system on $G\times^P(\mc{O}_{prin}+\mf{u}_P)$. 
 \[\begin{tikzcd}
G\times^P (\mathcal{O}_{prin}+\mathfrak{u}_P) \arrow[r, "c"]                                & \mathcal{N}_G     \\
G\times^P(\mathcal{O}_{prin}+\mathfrak{u}_P)\cap \mu^{-1}(x) \arrow[r, "c"] \arrow[u, hook] & x \arrow[u, hook]
\end{tikzcd}\]
 Using the above diagram,  $\Ind^G_P\mc{IC}(\mc{O}_{prin},\mc{L})_x$  becomes the $!$-pushforward of a local system on  $G\times^P(\mc{O}_{prin}+\mf{u}_P)\cap \mu^{-1}(x)$ by a constant map. From the table above,  $G\times^P(\mc{O}_{prin}+\mf{u}_P)\cap \mu^{-1}(x)$ is simply connected for $\mc{O}[4]$, $\mc{O}[2,1^2]$ and $\mc{O}[1^4]$. A local system on a simply connected space is constant sheaf. Therefore for these orbits, each stalk of   $\Ind^G_P\mc{IC}(\mc{O}_{prin},\mc{L})$ is the cohomology  of  $G\times^P(\mc{O}_{prin}+\mf{u}_P)\cap \mu^{-1}(x)$. But for $\mc{O}[2^2]$, 
 $G\times^P(\mc{O}_{prin}+\mf{u}_P)\cap \mu^{-1}(x)$ is $\mb{A}^1-\{pt\}$, which is not simply connected. Here we will abuse the notation  little bit, both the representative of $\mc{O}[2^2]$  and its image under the projection $\mc{N}_P\to \mc{N}_L$ will be called $x$. Recall we started with a nontrivial $L$-equivariant local system on $\mc{O}_{prin}$. The projection $\pi:\mc{O}_{prin}+\mf{u}_P \to \mc{O}_{prin}$ is a trivial vector bundle, hence induces isomorphism of the equivariant fundamental groups. The inclusion $\mc{O}_{prin}+\mf{u}_P \hookrightarrow G\times ^P(\mc{O}_{prin}+\mf{u}_P)$ induces isomorphism on the equivariant fundamental groups via induction equivalence.  So the pullback of the local system we started with is  still a nontrivial local system on $G\times ^P(\mc{O}_{prin}+\mf{u}_P)$. Let 
 \[S=\{\begin{pmatrix}
A &	0\\ 0& A
\end{pmatrix}| \hspace{2mm} A=\begin{pmatrix}
a&b\\-b&a	
\end{pmatrix}, a,b\in \mb{C}
, a^2+b^2=1 \}.\] It is not hard to check $S\subset G^x$. If we choose a flag in $\mu^{-1}(x)\cap G\times^P(\mc{O}_{prin}+\mf{u}_P)$, say $F=\langle e_1 \rangle \subset e_1, e_2, e_4 \rangle$ then we can see the elements of $S$ that fix the flag $F$ are $\{\begin{pmatrix}
Id& 0\\ 0& Id	
\end{pmatrix}, \begin{pmatrix}
 -Id &0 \\ 0& -Id	
 \end{pmatrix}\}$. Therefore,
 
 \[S^F/(S^F)^\circ = \{\begin{pmatrix}
Id& 0\\ 0& Id	
\end{pmatrix}, \begin{pmatrix}
 -Id &0 \\ 0& -Id	
 \end{pmatrix}\}\cong \mb{Z}/2\mb{Z}.\] Now we aim to  show $S^F/(S^F)^\circ \cong L^x/(L^x)^\circ$.  It is not hard to see \[L^x \cong \{\begin{pmatrix}
a&&&\\&b&&\\&&a^{-1}&\\&&&b	
\end{pmatrix}| \hspace{2mm} a\in \mb{C}, b=\pm 1
\}\cong G_m\times\{\pm 1\}.\]Hence $L^x/(L^x)^\circ\cong \mb{Z}/2\mb{Z}$ and the map $S^F/(S^F)^\circ \to L^x/(L^x)^\circ$ is an isomorphism. Therefore $(b^*\f^G_P)^{-1}a^*\mc{L} [2]|_{\mu^{-1}(x)\cap G\times^P(\mc{O}_{prin}+\mf{u}_P)}$ is  a nontrivial local system. 

For  a connected, locally contractible space $X$ if  the universal cover is contractible, then the inclusion functor $\lc(X,\Bbbk) \to Sh(X,\Bbbk)$ induces an equivalence of categories, 
 \[D^b\lc(X,\Bbbk) \to D^b_{loc}(X,\Bbbk).\] 
 
 It can be proved by a minor variation on the proof that the (co)homology of an Eilenberg-MacLane space is isomorphic to group (co)homology \cite[Prop ~II.4.1]{Bro}.
 For $X=\mb{A}^1-\{pt\}$, $\lc(X,\Bbbk)\cong \Bbbk[\pi_1(\mb{A}^1-\{pt\})]-mod$, which is same as $\Bbbk[\mb{Z}]-mod\cong \Bbbk[T,T^{-1}]-mod$. Therefore, for the local system $(b^*\f^G_P)^{-1}a^*\mc{L} [2] |_{\mu^{-1}(x)\cap G\times^P(\mc{O}_{prin}+\mf{u}_P)} $, there exists a $\Bbbk[T,T^{-1}]$ module $M$ on which $T$ acts by $(-1)$. To calculate the cohomology of this local system is same as calculating $R\Hom(\Bbbk, M)$. Now,\[\xrightarrow{}\Bbbk[T,T^{-1}]\xrightarrow{\times(T-1)} \Bbbk[T,T^{-1}] \xrightarrow{T \to 1} \Bbbk \hspace{2mm}\]
is a projective resolution of $\Bbbk$. Applying  $\Hom( ,M)$ we get, $M$ in degree $0$ and $1$.
\[\to 0\to M\xrightarrow{\times(-2)} M \to 0 \to.\]Multiplying by $-2$ induces isomorphism. Hence $R\Hom(\Bbbk,M)$ is $0$ in every degree.

\FloatBarrier
 \begin{table}[h!]
 \caption{Stalks of $\Ind^G_P\mc{IC}(\mc{O}_{prin},\mc{L})$}
\begin{center}
\begin{tabular}{|c|c|c|c|c|}
\hline
 $\dim$& $\mc{O}[4]$& $\mc{O}[2^2]$ & $\mc{O}[2,1^2]$ & $\mc{O}[1^4]$\\
 $0$&&&&\\$-1$&&&&\\$-2$&&&rank 1&\\$-3$&&&&\\$-4$&&&rank 1&\\$-5$&&&&\\$-6$&&&&\\$-7$&&&&\\$-8$&&&&\\$-9$&&&&\\$-10$& rank 1&&&\\
 \hline
\end{tabular}	
\end{center}
\end{table}
 \FloatBarrier

 Hence the parity condition of Conjecture \ref{2.6(c)} is satisfied.
 
\end{example}
\begin{example}\textbf{$\mf{sl}_4$-case:}
 Let $G=SL_4$. First we talk about the Levi subgroups of $G$ and find out which of them have cuspidal pairs. 
 The conjugacy classes of proper Levis are of the form
 \[
 S(GL_3\times GL_1),S(GL_2\times GL_2), \\
 S(GL_2\times GL_1 \times GL_1), T.
 \]Here $S(GL_m\times GL_n)=\{\begin{pmatrix}
A&0\\0&B \end{pmatrix} | \hspace{2mm} A\in GL_m, B\in GL_n, \det(A)\det(B)=1	\}$.
  According to the discussion in 6.2 and Theorem 6.3 in \cite{AJHR3}, we can see  cuspidal pair only appears for $S(GL_2\times GL_2)$ and is of the form $( \mc{O}_{prin} \times \mc{O}_{prin}, \mc{L}\boxtimes \mc{L})$. Here each $\mc{L}$ is a rank one $SL(2)$-equivariant local system on  $\mc{O}_{prin}$ and   $\mc{O}_{prin}$  is the $SL_2$-principle nilpotent orbit in $\mf{sl}_2$. 
 
 For $\mf{sl}_4$, the root system is $\Phi=\{e_i-e_j| i\ne j,1\leq i,j\leq 4\}$. The parabolic subgroup associated to $\{e_1-e_2,e_3-e_4\}$ is of the form,
 $\begin{pmatrix}
 \ast &\ast &\ast & \ast \\	 \ast &\ast &\ast & \ast \\	&&\ast &\ast \\&&\ast & \ast \\
 \end{pmatrix}$. The Levi subgroup is then $S(GL_2\times GL_2)$  and the unipotent radical is of the form,
 $\begin{pmatrix}
 &&\ast &\ast\\&& \ast & \ast\\&&&\\&&&\\
 \end{pmatrix}$. Now the generators of the nilpotent orbits come from the Jordan block of size depending on the partition. Hence the representatives of $\mc{O}[4],\mc{O}[3,1],\mc{O}[2^2],\mc{O}[2,1^2],\mc{O}[1^4]$ are respectively,
 \[\begin{pmatrix}
0&1&0&0\\0&0&1&0\\0&0&0&1\\0&0&0&0\\	
\end{pmatrix},
\begin{pmatrix}
0&1&0&0\\0&0&1&0\\0&0&0&0\\0&0&0&0\\	
\end{pmatrix},
\begin{pmatrix}
0&1&0&0\\0&0&0&0\\0&0&0&1\\0&0&0&0\\	
\end{pmatrix},
\begin{pmatrix}
0&1&0&0\\0&0&0&0\\0&0&0&0\\0&0&0&0\\	
\end{pmatrix}, \{0\}.
\]
Now we calculate $\mu^{-1}(x)$ for each $x$ as we did for $\mf{sp}_4$.  Here again, \[G\times^P(\mc{N}_L+\mf{u}_P)=\{(gP,x)\in G/P\times \mc{N}_G|Ad(g^{-1})x\in \lie(P)\}.\]
But $Ad(g^{-1})x\in\mf{p}$ means it preserves the two dimensional subspace $\langle e_1, e_2 \rangle$. Hence, 
\[
G\times^P(\mc{N}_L+\mf{u}_P)=\{(H,x)| \hspace{2mm}x\in \mc{N}_G,   \hspace{1mm}H \text{ is a two dimensional subspace preserved by $x$}\},
\]  
\begin{lemma}
Let $x\in \mc{N}_L+\mf{u}_P$. Then, $x\in \mc{O}_{prin}\times \mc{O}_{prin}+\mf{u}_P$ if and only if $x|_{\langle e_1, e_2\rangle}\ne 0$ and $x|_{\mb{C}^4/{\langle e_1, e_2 \rangle}}\ne 0$.	
	\end{lemma}
The proof follows from some easy matrix calculations.
Therefore following the same process as for $\mf{sp}_4$, \[\mu^{-1}(x)\cap (\mc{O}_{prin}\times \mc{O}_{prin}+\mf{u}_P)=\{(H,x)|\hspace{2mm} x \text{ preserves the subspace $H$, } x|_H\ne 0,x|_{\mb{C}^4/H}\ne 0\} .\]

\begin{lemma}
Let $x$ be the representative of each orbits in $\mf{sl}_4$.\begin{enumerate}
\item For $\mc{O}[4]$, $\mu^{-1}(x)$ is $\{\langle e_1, e_2 \rangle\}$ and $\mu^{-1}(x)\cap (\mc{O}_{prin}\times \mc{O}_{prin}+\mf{u}_P)=\{\langle e_1, e_2 \rangle\}$.
	\item  For $\mc{O}[3,1]$, $\mu^{-1}(x)\cong \mb{P}^1$ and $ \mu^{-1}(x)\cap (\mc{O}_{prin}\times \mc{O}_{prin}+\mf{u}_P)\cong \mb{P}^1-\{[0,1],[1,0]\}$.
	\item For $\mc{O}[2,1^2]$, 
if $x$ preserves $H$, then either $\langle e_1\rangle \subset H$ or $H \subset \ker(x)$. Also 
$\mu^{-1}(x)\cong \mb{P}^2 \sqcup_{\mb{P}^1} \mb{P}^2$ and $ \mu^{-1}(x)\cap (\mc{O}_{prin}\times \mc{O}_{prin}+\mf{u}_P) \cong \emptyset$.
\item For $x\in \mc{O}[2^2]$, $\mu^{-1}(x)\cap (\mc{O}_{prin}\times \mc{O}_{prin}+\mf{u}_P)\cong \mb{A}^1\sqcup \mb{A}^2$.
\end{enumerate}	
\end{lemma}
\begin{proof}
\begin{enumerate}
\item For $\mc{O}[4]$, the only choice for $\mu^{-1}(x)$ is $\langle e_1, e_2 \rangle$. Now $x|_{\langle e_1, e_2\rangle}\ne 0$ and $x|_{\mb{C}^4/{\langle e_1, e_2 \rangle}}\ne 0$, hence $\mu^{-1}(x)\cap (\mc{O}_{prin}\times \mc{O}_{prin}+\mf{u}_P)=\{\langle e_1, e_2 \rangle\}$.	
\item The first claim is that  if $H\in \mu^{-1}(x)$, then $H$ contains $e_1$. Let $H$ does not contains $e_1$.
if $H$ contains $v=ae_1+be_2+ce_3+de_4$, then as $x.v$ is in $H$, so $be_1+ce_2\in H$. If both $b=c=0$, then $v\in \ker( x)=\langle e_1, e_4 \rangle$. If $v \ne e_1$ then it is a linear combination of $e_1$ and $e_4$. In this case  the other basis element of $H$ must be a linear combination of $e_2$ and $e_3$, which contradicts the fact that $H$ $x$-invariant. So $b$ and $c$ both can not be $0$. If $c=0$ we are done. If $c\ne 0$ then $x.(x.v)=ce_1\in H$, therefore $e_1\in H$. 

 Now as $e_1$ is fixed, we have one choice left for the second generator. Now let the second generator $v= ae_2+be_3+de_4$, then $x.v=ae_1+ce_2$. As $H$ is stable under $x$, so $x.v$ must be a scalar multiple of $e_1$ or $v$. In both cases $c=0$, therefore $v$ can be linear combination of $e_2$ and $e_4$. So we have $\langle e_1 \rangle \subset H \subset \langle e_1, e_2, e_4 \rangle$ and $\mu^{-1}(x)\cong \mb{P}^1$.  
 
 For $ \mu^{-1}(x)\cap (\mc{O}_{prin}\times \mc{O}_{prin}+\mf{u}_P) $, $x|_H\ne 0$. So again if we take the other generator $v$ in $H$, which is of the form $ae_2+be_4$, then using the condition $x|_H\ne 0$ we can say $a\ne 0$. Now if $b=0$, then $H=\langle e_1, e_2 \rangle$, which implies $x|_{\mb{C}^4/H}=0$. Therefore  for $H$ to be in  $ \mu^{-1}(x)\cap (\mc{O}_{prin}\times \mc{O}_{prin}+\mf{u}_P)$, $a$ and $b$ both must be nonzero. So $ \mu^{-1}(x)\cap (\mc{O}_{prin}\times \mc{O}_{prin}+\mf{u}_P) \cong \mb{P}^1-\{[0,1], [1,0]\}$.
 
 \item Let $H$ is not contained in $\ker(x)=\langle e_1, e_3, e_4 \rangle$. Let $v=ae_1+be_2+ce_3+de_4\in H$ with $b\ne 0$, then   $x.v=be_1\in H$. Hence $H$ contains $e_1$. Now if $\langle e_1 \rangle \subset H \subset \mb{C}^4$, then the choice for $H$ is $\mb{P}^2$. If $H\subset \ker(x)=\langle e_1, e_3, e_4\rangle$, then again choice is $\mb{P}^2$. If $\langle e_1 \rangle \subset H \subset \ker(x)$,  then the choice is $\mb{P}^1$. Hence $\mu^{-1}(x)\cong \mb{P}^2\sqcup _{\mb{P}^1}\mb{P}^2$. For $ \mu^{-1}(x)\cap (\mc{O}_{prin}\times \mc{O}_{prin}+\mf{u}_P) $, $x|_H \ne 0$, so $H$ can not be contained in $\ker(x)$. But still whatever be the choice of the other generator we can see $x|_{\mb{C}^4/H}$ is always $0$. Therefore $ \mu^{-1}(x)\cap (\mc{O}_{prin}\times \mc{O}_{prin}+\mf{u}_P)=\emptyset $.
 \item If $ae_1+be_2+ce_3+de_4 \in H$, then $be_1+de_3 \in H$. If $x|_H\ne 0$, then $H\cap \ker(x)$ is one dimensional. Call the subspace  $\ker(x)\cap H$ to be $L$ which is definitely generated by elements of the form $be_1+de_3$. Now clearly $L \subset H \subset x^{-1}L=\langle e_1, e_3, be_2+de_4\rangle$.
	If both $b$ and $d$ are $0$, then $H=\langle e_1, e_3\rangle$. This implies $x|_H=0$. The converse is also true.
	For $H$ to be in $\mu^{-1}(x)\cap (\mc{O}_{prin}\times \mc{O}_{prin}+\mf{u}_P)$ we need $x|_H\ne 0$ and $x|_{\mb{C}^4/H}\ne 0$, which is not true for the above case. So one of them must be non-zero. Now if we consider one of them is zero, say $d=0$, then $L=\langle e_1\rangle$. In this case we can consider $b=1$, hence $H$ is generated by $e_1$ and $e_2+ce_3$, and in this case $x|_{\mb{C}^4/H}\ne 0$, so the choice is $\mb{A}^1$. 
	The remaining case is $d\ne0$. Here we can consider $d=1$ and then $L=\langle be_1+e_3\rangle$. In this case, $H$ is generated by $be_1+e_3$ and $ae_1+be_2+ce_3+e_4$, which is the same as $\langle be_1+e_3, a'e_1+be_2+e_4\rangle$, so the choice is $\mb{A}^2$. In this case also $x$ is nonzero on both $H$ and the quotient.

\end{enumerate}
	
\end{proof}

\FloatBarrier
\begin{table}[h!]
\caption{\textbf{Orbits in $\mf{sl}_4$}}
\begin{center}
\begin{tabular}{|c|c|c|c|c|c|c|}
\hline
orbits:& $\mc{O}[4]$ & $\mc{O}[3,1]$&$\mc{O}[2^2]$ & $\mc{O}[2,1^2]$ & $\mc{O}[1^4]$ \\ \hline
$\dim:$ &12& 10&8&6&0\\
\hline
$\pi_1:$& $\mb{Z}/4$ & $\{1\}$ &  $\mb{Z}/2$ &   $\{1\}$ & $\{1\}$\\
\hline
$\mu^{-1}(x): $& $\{\langle e_1,e_2\rangle\}$&$\mb{P}^1$  & -& $\mb{P}^2\sqcup _{\mb{P}^1}\mb{P}^2$ & $G/P$\\
\hline
$\mu^{-1}(x)\cap G\times^P(\mc{O}_{prin}\times \mc{O}_{prin} +\mf{u}_P):$ & $\{\langle e_1, e_2 \rangle\} $& $\mb{P}^1-\{[0,1],[1,0]\}$ & $\mb{A}^2\sqcup\mb{A}^1$ & $\emptyset$ & $\emptyset$ 
\\
\hline
\end{tabular}
\end{center}
\end{table}
\FloatBarrier

Now we are ready to find the $\Ind^G_P$. The only cuspidal pair in $L$ is $(\mc{O}_{prin}\times \mc{O}_{prin},\mc{L}\boxtimes \mc{L})$ where $\mc{L}$ is the nontrivial local system on $\mc{O}_{prin}$. We know $\mc{IC}(\mc{O}_{prin}\times \mc{O}_{prin},\mc{L}\boxtimes \mc{L})\cong \mc{IC}(\mc{O}_{prin},\mc{L})\boxtimes \mc{IC}(\mc{O}_{prin},\mc{L})$.  We can use the parabolic induction diagram introduced in \ref{subsec5.2}, therefore $\Ind^G_P \mc{IC}  (\mc{O}_{prin}\times \mc{O}_{prin},\mc{L}\boxtimes \mc{L})=c_!(b^*\f^G_P)^{-1}a^*(\mc{L}[2]\boxtimes \mc{L}[2])$. 
 Now we will follow the same steps as we did for $\mf{sp}_4$ and using the same diagram,  $\Ind^G_P\mc{IC}(\mc{O}_{prin}\times \mc{O}_{prin},\mc{L}\boxtimes \mc{L})_x$  becomes the $!$-pushforward of a local system on  $G\times^P(\mc{O}_{prin}\times \mc{O}_{prin}+\mf{u}_P)\cap \mu^{-1}(x)$ by a constant map. From the table above,  $G\times^P(\mc{O}_{prin}\times \mc{O}_{prin}+\mf{u}_P)\cap \mu^{-1}(x)$ is simply connected for $\mc{O}[4]$, $\mc{O}[2,1^2]$, $\mc{O}[2^2]$ and $\mc{O}[1^4]$. A local system on a simply connected space is constant sheaf. Therefore for these orbits, the stalks of  $\Ind^G_P\mc{IC}(\mc{O}_{prin}\times \mc{O}_{prin},\mc{L}\boxtimes \mc{L})$ are the cohomologies of  $G\times^P(\mc{O}_{prin}\times \mc{O}_{prin}+\mf{u}_P)\cap \mu^{-1}(x)$. But for $\mc{O}[3,1]$, 
 $G\times^P(\mc{O}_{prin}\times \mc{O}_{prin}+\mf{u}_P)\cap \mu^{-1}(x)$ is $\mb{P}^1-\{[1,0],[0,1]\}$, which is not simply connected. Here we will use the same abuse of notation, both the representative of $\mc{O}[3,1]$  and its image under the projection $\mc{N}_P\to \mc{N}_L$ will be called $x$. Recall we started with a nontrivial $L$-equivariant local system on $\mc{O}_{prin}\times \mc{O}_{prin}$. The projection $\pi:\mc{O}_{prin}\times \mc{O}_{prin}+\mf{u}_P \to \mc{O}_{prin} \times \mc{O}_{prin} $ is a trivial vector bundle, hence induces isomorphism of the equivariant fundamental groups. The inclusion $\mc{O}_{prin}   \times \mc{O}_{prin}+\mf{u}_P \hookrightarrow G\times ^P(\mc{O}_{prin}\times \mc{O}_{prin} +\mf{u}_P)$ induces isomorphism on the equivariant fundamental groups via induction equivalence.  So the pullback of the local system we started with is  still a nontrivial local system on $G\times ^P(\mc{O}_{prin}\times \mc{O}_{prin}+\mf{u}_P)$. Let 
 \[S=\{\begin{pmatrix}
A &	0\\ 0& A
\end{pmatrix}| \hspace{2mm} A=\begin{pmatrix}
a&\\&a^{-1}	
\end{pmatrix}, a\in \mb{C}^\times \}.\] and surely $S\subset G^x$. If we choose a subspace in $\mu^{-1}(x)\cap G\times^P(\mc{O}_{prin}\times \mc{O}_{prin} +\mf{u}_P)$, say $H=\langle e_1, e_2+e_4 \rangle$ then we can see $S$ stabilizes $H$, so $S^H=S$.  Therefore,
 $S^H/(S^H)^\circ$ is trivial. Now we aim to  show $S^H/(S^H)^\circ \cong L^x/(L^x)^\circ$.  It is not hard to see \[L^x \cong \{\begin{pmatrix}
a&&&\\&b&&\\&&a^{-1}&\\&&c&b^{-1}	
\end{pmatrix}| \hspace{2mm} a,b\in \mb{C}^\times \text{ and }c \in \mb{C}
\}\cong \mb{C}^\times \times \mb{C}^\times \times \mb{C}.\]Hence $L^x/(L^x)^\circ$ is also trivial and the map $S^H/(S^H)^\circ \to L^x/(L^x)^\circ$ is an isomorphism. Therefore $(b^*\f^G_P)^{-1}a^*(\mc{L} [2]\boxtimes \mc{L}[2])|_{\mu^{-1}(x)\cap G\times^P(\mc{O}_{prin}\times \mc{O}_{prin}+\mf{u}_P)}$ is  a nontrivial local system. Now again using the same argument from \cite[Prop. ~II.4.1]{Bro},\[D^b\lc(X,\Bbbk) \to D^b_{loc}(X,\Bbbk).\]Here $X=\mb{P}^1-\{[1,0],[0,1]\}$, $\lc(X,\Bbbk)\cong \Bbbk[\pi_1(\mb{P}^1-\{[0,1],[1,0]\})]-mod$, which is same as $\Bbbk[\mb{Z}]-mod\cong \Bbbk[T,T^{-1}]-mod$. Therefore, for the local system $(b^*\f^G_P)^{-1}a^*\mc{L} [2]\boxtimes \mc{L}[2] |_{\mu^{-1}(x)\cap G\times^P(\mc{O}_{prin}\times \mc{O}_{prin}+\mf{u}_P)} $, there exists a $\Bbbk[T,T^{-1}]$ module $M$ on which $T$ acts by $(-1)$. Now we use the same calculation as we did for $\mf{sp}_4$ to conclude  $R\Hom(\Bbbk,M)$ is $0$ in every degree.

 \FloatBarrier
\begin{table}[h!]
 \caption{Stalks of $\Ind^G_P\mc{IC}(\mc{O}_{prin}\times \mc{O}_{prin},\mc{L}\boxtimes \mc{L})$}
\begin{center}
\begin{tabular}{|c|c|c|c|c|c|}
\hline
 $\dim$& $\mc{O}[4]$&$\mc{O}[3,1]$ &$\mc{O}[2^2]$ & $\mc{O}[2,1^2]$ & $\mc{O}[1^4]$\\
 $-6$&&&&&\\$-7$&&&&&\\$-8$&&&&rank 1&\\$-9$&&&&&\\$-10$&&&rank 1&&\\$-11$&&&&&\\$-12$&& &&&\\$-13$&&&&&\\$-14$&rank 1&&&&\\$-15$&&&&&\\$-16$& &&&&\\
 \hline
\end{tabular}	
\end{center}
\end{table}
\FloatBarrier

Hence the parity condition is again satisfied.
\end{example}
\subsection {Construction of $\mf{p},\mf{n},\text{ and } \mf{l}$ described in \ref{subsec6.1}:}\label{subsec8.1}
\begin{example}

Let $G=SL_4$ and $\chi: \mb{C}^\times \to G$ be defined as $t  \to (t,1,1,t^{-1})$. Then the matrix that gives $\mf{g}_{m'}$ for all $m'$ is 
\begin{equation}\label{32}
 \begin{pmatrix}
 0 & 1&1&2\\-1 &0&0&1\\ -1 &0&0&1\\
 -2&-1&-1&0\\
 \end{pmatrix}.
\end{equation} 
Now $\mf{g}_{m'}$ comes from the above matrix by putting nonzero entries wherever we have $m'$ in (\ref{32}) and $0$ elsewhere. For example,
 
 \[\mf{g}_2=
\begin{pmatrix}

 &  &\begin{matrix}  \ast  \end{matrix}\\ & &\\& &\\
\end{pmatrix}
, \mf{g}_1= 
\begin{pmatrix} & \ast & \ast\\
&&&\begin{matrix} \ast  \\ \ast  \end{matrix}\\ \\
\end{pmatrix}, \text{ and}\]
\[
\mf{g}_0= 
\begin{pmatrix}
\begin{matrix} \ast \end{matrix}& & \\ & \begin{matrix} \ast & \ast   \\	\ast  & \ast \end{matrix}
\\
 & & \begin{matrix}    \ast 
\end{matrix}
\end{pmatrix}.
\]Similarly we can find $\mf{g}_{-1},\mf{g}_{-2}$.
\\
Now  choose a point $x\in \mf{g}_{-1}$. We can think of $\mf{g}_{-1}$ as $\Hom(\mb{C},\mb{C}^2) \times \Hom(\mb{C}^2,\mb{C})$, which is a space of representations of quivers of finite type of dimension $(1,2,1)$. By \cite[Theorem ~4.3.9]{DW}, isomorphism classes of $G_0$-orbits in $\mf{g}_{-1}$ are in bijection with the isomorphism classes  of  finite type quiver representations of dimension $(1,2,1)$. This is again a linear combinition of roots of $A_3$ that add up-to $\alpha_1+2\alpha_2+\alpha_3$, where $\alpha_1,\alpha_2,\alpha_3$ are all the simple roots in $A_3$ and $\alpha_4=\alpha_1+\alpha_2, \alpha_5=\alpha_1+\alpha_2+\alpha_3,\alpha_6=\alpha_2+\alpha_3$. Let us pick one such linear combination which  gives a representative of that orbit and  call it $x$. Note that $\al_4+\al_6$  adds up-to the desired sum. This gives the representative \[x=
\begin{pmatrix}
 &&&\\1&&&\\0&&&\\&0&1& 
\end{pmatrix}
\]Note that the Jordan canonical form of $x$ is
\[ 
\begin{pmatrix}
0&1&&\\0&0&&\\&&0&1\\&&0&0\\	
\end{pmatrix}
\]This matrix is associated to the partition $[2,2]$, hence we can use \cite[Lemma ~3.2.6]{CM} to find a map  $ \mf{sl}_2 \to \mf{g}$ which takes  $e$ to  $ \begin{pmatrix}
0&1&&\\0&0&&\\&&0&1\\&&0&0\\	
\end{pmatrix} $ and  $h$ to 
$\begin{pmatrix}
\\ 1&&&\\&-1&&\\&&1& \\&&&-1\\	
\end{pmatrix}\in \mf{g}_0$.
We can see the matrix $\begin{pmatrix}
	&&&1\\&&1&\\&1&&\\1&&&\\
\end{pmatrix}$ conjugates $x$ to $\begin{pmatrix}
0&1&&\\0&0&&\\&&0&1\\&&0&0\\	
\end{pmatrix}$, therefore conjugating the above map by the same matrix we get a map $\phi: \mf{sl}_2\mapsto \mf{g}$ which sends $e$ to $x$ and $h$ to
\[
\begin{pmatrix}
	-1&&&\\&1&&\\&&-1&\\&&&1\\\
\end{pmatrix}\in \mf{g}_0
\]
Clearly, $\tilde{\phi}\begin{pmatrix}
t&\\&t^{-1}	
\end{pmatrix}= \begin{pmatrix}
	t^{-1}&&&\\ & t&&\\&&t^{-1}& \\&&&t \\ 
\end{pmatrix}.
$
Hence we get the required $\chi':\mb{C}^\times \to G$, as $\chi'(t)=\tilde\phi 
\begin{pmatrix}
t&\\&t^{-1}	
\end{pmatrix}= \begin{pmatrix}
	t^{-1}&&&\\ & t&&\\&&t^{-1}& \\&&&t \\ 
\end{pmatrix}
$. 
 	So the matrix $_m \mf{g}$ for all $m$ comes from the matrix below by the same procedure as above.

\begin{equation}\label{33}                                                                                                                                                    \begin{pmatrix}                                                                                                                                              0&-2&0&-2	\\2&0&2&0	\\0&-2&0&-2	\\2&0&2&0                                                                                                                                                   \\

                                                                                                                              \end{pmatrix}.
                                                                                                                                                    \end{equation}
We can see what  $_1\mf{g}, _2\mf{g},_3\mf{g}$ are as before. In this example $n=-1$. So we need conditions on $m+2m'$ to find $\mf{p,n,l}$. The matrix $_m\mf{g}_{m'}$ for $m+2m'$ is given below.
 \[\begin{pmatrix}
	0&0&2&2\\0&0&2&2\\-2&-2&0&0\\-2&-2&0&0\\
\end{pmatrix}.
\]

Hence, with the conditions on $m+2m'$ we can say,
\[\mf{l}
=\begin{pmatrix}
	 \ast & \ast && \\ \ast & \ast && \\&& \ast & \ast\\ &&\ast & \ast \\
\end{pmatrix}.\]
\FloatBarrier
\begin{table}[h!]
\caption{Table for the Levis}
\begin{center}
\begin{tabular}{|c|c|c|c|}
\hline
 $x$& Representative&$\dim$& Associated Levi\\
 \hline
 $\al_2+\al_3+\al_4$ &$\begin{pmatrix}
&&&\\0&&&\\1&&&\\&0&0&\\	
\end{pmatrix}
$&2 &$\begin{pmatrix}
\ast&\ast&\ast&\\\ast&\ast &\ast&\\\ast&\ast&\ast&\\&&&\ast	
\end{pmatrix}
$ \\ $\al_6+\al_4$&$\begin{pmatrix}
&&&\\1&&&\\0&&&\\&0&1&\\	
\end{pmatrix}
$&3&$\begin{pmatrix}
	\ast&\ast&&\\\ast&\ast &&\\&&\ast&\ast\\&&\ast&\ast
\end{pmatrix}
$ \\$\al_2+\al_5$&$\begin{pmatrix}
&&&\\0&&&\\1&&&\\&0&1&\\	
\end{pmatrix}
$&4&$\begin{pmatrix}
\ast&\ast&&\\\ast &\ast&&\\&&\ast&\\&&&\ast\\	
\end{pmatrix}
$ \\$\al_1+\al_2+\al_6$&$\begin{pmatrix}
&&&\\0&&&\\0&&&\\&0&1&\\	
\end{pmatrix}
$& 2&$\begin{pmatrix}
	\ast&&&\\&\ast &&\\&&\ast&\ast\\&&\ast&\ast
\end{pmatrix}

$ \\$\al_1+2\al_2+\al_3$&$\{0\}$ & 0& $G$\\
 \hline
\end{tabular}	
\end{center}
\end{table}
\FloatBarrier
\end{example}

\end{document}